\documentclass[english]{amsart}

\usepackage{esint}
\usepackage[svgnames]{xcolor} 
\usepackage{color}
\usepackage[colorlinks,citecolor=red,pagebackref,hypertexnames=false,breaklinks]{hyperref}
\usepackage{pgf,tikz}
\usepackage{pdfsync}

\usepackage{dsfont}
\usepackage{url}
\usepackage[utf8]{inputenc}
\usepackage[T1]{fontenc}
\usepackage{lmodern}
\usepackage{babel}
\usepackage{mathtools}  
\usepackage{amssymb}
\usepackage{lipsum}
\usepackage{mathrsfs}
\usepackage{color}
\usepackage[skip=2.2pt plus 1pt, indent=12pt]{parskip}
\usepackage{stmaryrd}
\usepackage{soul}
\usepackage{ulem}

\newtheorem{theorem}{Theorem}[section]
\newtheorem{proposition}{Proposition}[section]
\newtheorem{lemma}{Lemma}[section]

\newtheorem{corollary}{Corollary}[section]

\newtheorem{remark}{Remark}[section]
\newtheorem{assumption}{Assumption}[section]

\numberwithin{equation}{section}

\title[Spectral Barron spaces]{Functional Analysis and Partial Differential Equations in Spectral Barron Spaces}

\author{Mourad Choulli}
\address{Universit\'{e} de Lorraine, 34 cours L\'{e}opold, 54052 Nancy cedex, France}
\email{mourad.choulli@univ-lorraine.fr}

\author{Shuai Lu}
\address{School of Mathematics Sciences, SKLCAM and LMNS, Fudan University, 220 Handan Road, Shanghai 200433, China}
\email{slu@fudan.edu.cn}

\author{Hiroshi Takase}
\address{Department of Mathematics, Okayama University, 3-1-1 Tsushima-naka, Kita-ku, Okayama 700-8530, Japan}
\email{takase@math.okayama-u.ac.jp}

\date{}

\begin{document}
\begin{abstract}
Spectral Barron spaces, constituting a specialized class of function spaces that serve as an interdisciplinary bridge between mathematical analysis, partial differential equations (PDEs), and machine learning, are distinguished by the decay profiles of their Fourier transform. In this work, we shift from conventional numerical approximation frameworks to explore advanced functional analysis and PDE theoretic perspectives within these spaces. Specifically, we present a rigorous characterization of the dual space structure of spectral Barron spaces, alongside continuous embedding in H\"older spaces established through real interpolation theory. Furthermore, we investigate applications to boundary value problems governed by the Schr\"odinger equation, including spectral analysis of associated linear operators. These contributions elucidate the analytical foundations of spectral Barron spaces while underscoring their potential to unify approximation theory, functional analysis, and machine learning.
\end{abstract}



\maketitle

\tableofcontents

\section{Introduction}

Spectral Barron spaces constitute a specialized subset of function spaces that have recently emerged as a focal point of interdisciplinary research, acting as an important bridge between mathematical analysis, the theory of partial differential equations (PDEs), and the field of machine learning (see e.g., \cite{Ba, CLLZ, LuM, MM, SX, Xu} and references therein). Functions within these spectral Barron spaces are distinguished by their Fourier transform. The very definition of these spaces hinges on the imposition of specific conditions of integrability of these Fourier transforms, which, in turn, reflect the smoothness characteristics of the functions they represent.

The increasing interest in such spaces was significantly catalyzed by Barron's seminal work \cite{Ba}. In this foundational study, it was demonstrated that any function \( g \) defined on a bounded domain, given that it meets a certain condition pertaining to its Fourier transform, can be efficiently approximated in the \( L^2 \) norm by a two layer neural network. Specifically, if a function \( g \) is an element of a spectral Barron space, it thereby adheres to the following criterion:
\[ C_g := \int_{\mathbb{R}^n} |\xi||\hat{g}(\xi)|  d\xi < \infty, \]
where \( \hat{g} \) denotes the Fourier transform of \( g \), then \( g \) can be approximated by a two layer neural network \( g_m \) of the form
\[ g_m(x) = \sum_{j = 1}^m a_j \sigma(w_j \cdot x + b_j) + c.  \]
Here, \( \sigma \) represents the activation function. Notably, the number of neurons \( m \) exhibits at most polynomial dependence on the dimension \( n \), under the assumption that \( C_g \) is treated as a constant. As elucidated above, the weight coefficients \( a_j \) and \( w_j \), in conjunction with the bias coefficients \( c \) and \( b_j \) in the outer and inner layers of the neural network \( g_m \), play crucial roles in governing its approximation properties.
The quantitative universal approximation result established in \cite{Ba} shows that for any function \( g \) defined on a bounded domain and satisfying the aforementioned decay properties of its Fourier transform, there exists a shallow neural network \( g_m \) that approximates \( g \) uniformly with an error bound of \( O\left(\frac{1}{\sqrt{m}}\right) \).

In recent years, there has been a growing focus in the study of spectral Barron spaces, promoted by the global explosion of interest in deep learning and deep neural networks. The previously established plain decay property has been extended to a more comprehensive framework. In this new framework, the following decay condition is taken into account:
\[ C_{g,\varphi} := \int_{\mathbb{R}^n} \varphi(\xi) |\hat{g}(\xi)|  d\xi < \infty, \]
where \( \varphi(\xi) \) is a general weight function. This weight function can take various forms, such as \( \varphi(\xi) := (1 + |\xi|)^s \) or \( \varphi(\xi) := (1 + |\xi|^2)^{s/2} \), cf. \cite{CLLZ,LuM,SX,Xu}. In this context, the smoothing index \( s \) highlights the smoothness of the function, with larger values of \( s \) corresponding to smoother functions.
Recent studies, such as those in \cite{CLLZ, LuM, MM, Xu}, have demonstrated that for specific smoothness indices \( s \), the quantitative universal approximation rate  \( O\left(\frac{1}{\sqrt{m}}\right) \) remains valid under different decay conditions. These findings have sparked a growing interest in unraveling the relationships between functional decay behaviors and the approximation capabilities of neural networks. It is noteworthy that the latter weight function \( \varphi(\xi)=(1 + |\xi|^2)^{s/2} \) aligns with the Bessel potential function spaces introduced in the seminal work \cite{Ho}. The Bessel potential of order \( s \) in the context of distributions, defined as \( (I - \Delta)^{-s/2}\delta(x) \), can be formally defined via its Fourier transform. For a function \( f \) belonging to an appropriate function space, the Bessel potential of \( f \) is expressed as
\[ \mathscr{F}\left((1 - \Delta)^{-s/2}f\right)(\xi) = (1 + |\xi|^2)^{-s/2}\mathscr{F}f(\xi). \]
This expression precisely aligns with the form of the weight function mentioned above.

In the current work, while the majority of existing research concerning spectral Barron spaces has predominantly centered around numerical approximation or optimization within the realm of neural networks, our focus diverges towards functional analysis and PDEs in the context of spectral Barron spaces. Specifically, our objective is to identify the types of mathematical problems that can be naturally described within the framework of spectral Barron spaces. Consequently, these problems are well-suited to be solved by (deep) neural networks, with the property of quantitative universal approximation naturally holding true. The fundamental properties of spectral Barron spaces, denoted as $B^s$ where $s \geq 0$, have already been thoroughly investigated and established through a series of research endeavors (for instance, see \cite{CLLZ, LM, LuM, MM} and the references cited therein). One of the primary objectives of the current study is to achieve a comprehensive understanding of these spaces by filling in the remaining gaps in their characterization. Our key contributions to this field are multifaceted. Firstly, we provide a detailed characterization of the dual space of $B^s$, thereby enhancing the theoretical framework surrounding these spaces. Secondly, we establish the compact embedding property for $B^s$ spaces, which is crucial for understanding the behavior of sequences within these spaces. Thirdly, we elucidate the relationship between spectral Barron spaces and real interpolation spaces, offering new insights into the interpolation theory of these functional spaces. Furthermore, we demonstrate the embedding of $B^s$ spaces into H\"older spaces, thereby bridging the gap between these two important classes of function spaces. Similar to \cite{LuM}, another finding of our work is that the Laplacian operator defines a sectorial operator on $B^0$. This property is instrumental in enabling functional calculus within spectral Barron spaces, thus opening up new avenues for the analysis of partial differential equations in these spaces. Leveraging these novel properties of spectral Barron spaces, we proceed to investigate the Schr\"odinger equation in the entire space. We also explore boundary value problems (BVPs) within bounded domains, providing a more comprehensive understanding of the behavior of solutions to these equations in spectral Barron spaces. Moreover, we conduct a spectral analysis of the relevant operators, shedding light on their spectral properties. 

\textbf{Main notations.} 
For the sake of brevity and simplicity in notation, we shall employ abbreviated forms such as the Schwartz space $\mathscr{S}$ and its dual space, the
space of tempered distributions $\mathscr{S}'$, the Banach space $L^p$ with $p\geq 1$, and so forth, instead of the more cumbersome $\mathscr{S}(\mathbb{R}^n)$, $\mathscr{S}'(\mathbb{R}^n)$, $L^p(\mathbb{R}^n)$, etc. 
The Fourier transform, which will be utilized throughout our discussion, is defined for a function $\varphi$ belonging to $\mathscr{S}$ as follows:
\[
\hat{\varphi}(\xi) = \mathscr{F}\varphi(\xi) := \int_{\mathbb{R}^n} e^{-i x \cdot \xi} \varphi(x) \, dx, \quad \xi \in \mathbb{R}^n.
\]
The inverse Fourier transform, denoted by $\mathscr{F}^{-1}$, is given by
\[
\mathscr{F}^{-1}\varphi(\xi) = (2\pi)^{-n} \int_{\mathbb{R}^n} e^{i x \cdot \xi} \varphi(x) \, dx, \quad \xi \in \mathbb{R}^n, \; \varphi \in \mathscr{S}.
\]
In cases where there is no risk of ambiguity, we will also use the notations $\mathscr{F}$ and $\mathscr{F}^{-1}$ to represent the Fourier transform and its inverse when acting on the space of tempered distributions $\mathscr{S}'$. We recall that the Fourier transform on $\mathscr{S}'$ is defined via the duality pairing as follows:
\[
\langle \hat{f}, \varphi \rangle = \langle \mathscr{F}f, \varphi \rangle := \langle f, \hat{\varphi} \rangle, \quad f \in \mathscr{S}', \; \varphi \in \mathscr{S}.
\]
Here, and henceforth throughout this work, the notation $\langle \cdot, \cdot \rangle$ is employed to denote the duality pairing between $\mathscr{S}$ and $\mathscr{S}'$. 
In the subsequent discussion, the norm associated with the Lebesgue space $L^p$ for $1 \leq p \leq \infty$ will be represented as $\|\cdot\|_p$. Additionally, we define the space $C_b^0$ as  the space of bounded continuous functions, i.e., $C_b^0 := L^{\infty} \cap C^0$. 
 
We will also adopt the following standard notations throughout our discussion: for a multi-index $\alpha = (\alpha_1, \ldots, \alpha_n) \in (\mathbb{N}\cup\{0\})^n$ and a vector $\xi = (\xi_1, \ldots, \xi_n) \in \mathbb{R}^n$, we define $|\alpha| := \sum_{j=1}^{n} \alpha_j$, $\xi^\alpha := \xi_1^{\alpha_1} \ldots \xi_n^{\alpha_n}$, and $\partial^\alpha := \partial_1^{\alpha_1} \ldots \partial_n^{\alpha_n}$. Here, $\partial_j$ denotes the partial derivative with respect to the $j$-th variable. 
Additionally, we will use the notation $\langle \xi \rangle := \sqrt{1 + |\xi|^2}$, where $|\xi|^2 = \sum_{j=1}^{n} \xi_j^2$, throughout this text. 
Furthermore, when dealing with two Banach spaces $X$ and $Y$, the notation $Y \hookrightarrow X$ will be used to indicate that $Y$ is continuously embedded in $X$.

 \section{Spectral Barron spaces and their properties}
In this section, we shall present a comprehensive definition of spectral Barron spaces, along with an exploration of their detailed properties.

\subsection{Spectral Barron spaces $B^0$}

Before giving a precise definition of the spectral Barron space $B^0$, we first present some necessary preliminaries, which begin by proving the following lemma.
\begin{lemma}\label{lemB1}
Let $h\in L^1$. Then $\mathscr{F}^{-1}h\in C_b^0$, and
\begin{equation}\label{B1}
\mathscr{F}^{-1}h=\lim_{j\rightarrow\infty}\mathscr{F}^{-1}\varphi_j,
\end{equation}
where $(\varphi_j)$ is an arbitrary sequence in the Schwartz space $\mathscr{S}$ that converges to $h$ in the $L^1$ norm.
\end{lemma}

\begin{proof}
Since $\mathscr{S}$ is dense in $L^1$, there exists a sequence $(\varphi_j)$ in $\mathscr{S}$ such that $\varphi_j\rightarrow h$ in $L^1$ as $j\rightarrow\infty$. 
For  integers $j$ and $k$, we have the inequality
\[
\|\mathscr{F}^{-1}\varphi_j - \mathscr{F}^{-1}\varphi_k\|_{\infty}\leq(2\pi)^{-n}\|\varphi_j-\varphi_k\|_1.
\]
As $(\varphi_j)$ is a Cauchy sequence in $L^1$, the right-hand side of the above inequality tends to zero as $j,k\rightarrow\infty$. This implies that $(\mathscr{F}^{-1}\varphi_j)$ is a Cauchy sequence in $C_b^0$. Because $C_b^0$ is complete, there exists $\ell\in C_b^0$ such that $\lim_{j\rightarrow\infty}\mathscr{F}^{-1}\varphi_j = \ell$ in $C_b^0$. The continuity of the inverse Fourier transform on $\mathscr{S}'$ yields $\mathscr{F}^{-1}\varphi_j\rightarrow\mathscr{F}^{-1}h$ in $\mathscr{S}'$ and $\ell=\mathscr{F}^{-1}h$, which proves \eqref{B1}. 
\end{proof}

Let $f\in\mathscr{S}'$ such that $\hat{f}\in L^1$. This means there exists $g\in L^1$ with
\begin{equation}\label{1}
\langle \hat{f},\varphi\rangle=\langle g,\varphi\rangle,\quad \varphi\in\mathscr{S}.
\end{equation}
Equation \eqref{1} clearly leads to
\begin{equation}\label{2}
\langle \hat{f},\varphi\rangle=\langle f,\hat{\varphi}\rangle=\langle \mathscr{F}^{-1}g,\hat{\varphi}\rangle,\quad \varphi\in\mathscr{S}.
\end{equation}
Since the Fourier transform is an isomorphism on $\mathscr{S}$, \eqref{2} implies $f = \mathscr{F}^{-1}g$ in the $\mathscr{S}'$ sense. By Lemma \ref{lemB1}, $\mathscr{F}^{-1}g\in C_b^0$. So, we can identify the set $\{f\in\mathscr{S}';\;\hat{f}\in L^1\}$ with a subspace of $C_b^0$.

We define the spectral Barron space $B^0$ as
\[
B^0 := \{f\in C_b^0;\;\hat{f}\in L^1\},
\]
and endow $B^0$ with the norm
\[
\|f\|_{B^0}:=\|\hat{f}\|_1,\quad f\in B^0.
\]

Some key properties of the spectral Barron space $B^0$ are presented below.
\begin{lemma}\label{lem1}
The space $B^0$ is complete with respect to the norm $\|\cdot\|_{B^0}$.
\end{lemma}
\begin{proof}
Let $(f_j)$ be a Cauchy sequence in $B^0$. Then, by the definition of the norm in $B^0$, $(\hat{f}_j)$ is a Cauchy sequence in $L^1$. Since $L^1$ is complete, there exists $g\in L^1$ such that $\hat{f}_j\rightarrow g$ in $L^1$ as $j\rightarrow\infty$. In particular, $\hat{f}_j\rightarrow g$ in $\mathscr{S}'$. 

As the inverse Fourier transform $\mathscr{F}^{-1}$ is continuous on $\mathscr{S}'$, we have $f_j=\mathscr{F}^{-1}\hat{f}_j\rightarrow\mathscr{F}^{-1}g =: f$ in $\mathscr{S}'$. Also, since $\hat{f}\in L^1$, $f\in B^0$. Moreover, $\|f - f_j\|_{B^0}=\|\hat{f}-\hat{f}_j\|_1\rightarrow 0$ as $j\rightarrow\infty$. So, $(f_j)$ converges to $f$ in $B^0$ with respect to the norm $\|\cdot\|_{B^0}$.
\end{proof}

\begin{lemma}\label{lem1.1}
The Schwartz space $\mathscr{S}$ is dense in $B^0$.
\end{lemma}
\begin{proof}
Let $f\in B^0$ and $\epsilon>0$. Since $\mathscr{S}$ is dense in $L^1$, there exists $\psi\in \mathscr{S}$ such that $\|\hat{f}-\psi\|_{1}\leq\epsilon$. Let $\varphi:=\mathscr{F}^{-1}\psi$. Since $\psi\in \mathscr{S}$, $\varphi\in\mathscr{S}$. Thus we have $\|f - \varphi\|_{B^0}=\|\hat{f}-\psi\|_{1}\leq\epsilon$. 
\end{proof}

\begin{remark}\label{remB0}
{\rm
The spectral Barron space $B^0$ does not contain constant functions. To illustrate this, note that the Fourier transform of the Dirac measure $\delta$ is $\hat{\delta} = 1$.  Alternatively, by the Riemann-Lebesgue lemma, for any $h\in L^1$, the inverse Fourier transform $\mathscr{F}^{-1}h(x)$ converges to $0$ as $|x|\rightarrow\infty$, showing again that the spectral Barron space $B^0$ does not contain constant functions.
}
\end{remark}

\subsection{Spectral Barron spaces $B^s$ with $s\geq 0$ and the dual spaces}

We further define the generic spectral Barron space $B^s$ for $s\geq0$ as
\begin{align}\label{eq_SBarronSpace}
B^s := \{f\in C_b^0;\;\langle\xi\rangle^s\hat{f}\in L^1\},
\end{align}
where $\langle\xi\rangle:=\sqrt{1 + |\xi|^2}$. We equip $B^s$ with the natural norm
\begin{align}\label{eq_SBarronSpacenorm}
\|f\|_{B^s}:=\|\langle\xi\rangle^s\hat{f}\|_1,\quad f\in B^s.
\end{align}
Just as in the case of $B^0$, we can verify that $B^s$ is a Banach space when endowed with the norm $\|\cdot\|_{B^s}$. Let $0 \leq s \leq t$. It is straightforward to observe that the space $B^t$ is continuously embedded in $B^s$, denoted as $B^t \hookrightarrow B^s$, and for any function $f \in B^t$, the following inequality holds:
\[
\|f\|_{B^s} \leq \|f\|_{B^t}.
\]

\begin{remark}\label{remBs}
{\rm
For $s \geq 0$, define the measure $d\mu_s(\xi) := \langle \xi \rangle^s d\xi$. Under this definition, the norm in the space $B^s$ satisfies  
\[
\|f\|_{B^s} = \|\hat{f}\|_{L^1(\mathbb{R}^n, d\mu_s)}, \quad  f \in B^s.
\]  
Furthermore, for any measurable set $E \subseteq \mathbb{R}^n$, the equivalence $d\mu_s(E) = 0$ holds if and only if the Lebesgue measure $d\xi(E) = 0$. This establishes that $d\mu_s$ and the Lebesgue measure are mutually absolutely continuous.
}
\end{remark}

We define the pseudo-differential operator with the symbol $\langle \xi \rangle^{2s}$, denoted by $(1 - \Delta)^s$, as follows:
\[
(1 - \Delta)^s f(x) := (2\pi)^{-n} \int_{\mathbb{R}^n} e^{i x \cdot \xi} \langle \xi \rangle^{2s} \hat{f}(\xi) \, d\xi, \quad f \in B^{2s}.
\]
This definition can be equivalently expressed as
\begin{equation}\label{inv}
(1 - \Delta)^s f = \mathscr{F}^{-1}(\langle \xi \rangle^{2s} \hat{f}), \quad f \in B^{2s}.
\end{equation}
Consequently, we have
\[
\|(1 - \Delta)^s f\|_{B^0} = \|\langle \xi \rangle^{2s} \hat{f}\|_1 = \|f\|_{B^{2s}}, \quad f \in B^{2s}.
\]
This shows that $(1 - \Delta)^s$ is an isometric operator from $B^{2s}$ to $B^0$:
\begin{equation}\label{5}
\|(1 - \Delta)^s f\|_{B^0} = \|f\|_{B^{2s}}, \quad f \in B^{2s}.
\end{equation}
Now, let $g \in B^0$ and define $f = \mathscr{F}^{-1}(\langle \xi \rangle^{-2s} \hat{g})$. Then $f \in B^{2s}$, and by \eqref{inv}, we have $g = (1 - \Delta)^s f$. Moreover, $\|f\|_{B^{2s}} = \|g\|_{B^0}$. We conclude that $(1 - \Delta)^s$ defines an isometric isomorphism from $B^{2s}$ onto $B^0$. If $(1 - \Delta)^{-s}$ denotes the inverse of $(1 - \Delta)^s$, then we have
\[
\|(1 - \Delta)^{-s} g\|_{B^{2s}} = \|g\|_{B^0}, \quad g \in B^0.
\]
In particular, this implies that $B^s$ can be represented as the range of $(1 - \Delta)^{-s/2}$, denoted by
\[
B^s = R((1 - \Delta)^{-s/2}) := (1 - \Delta)^{-s/2}(B^0).
\]

For $s \geq 0$, we define the operator $K_s$ as follows:
\[
K_s f(x) := (2\pi)^{-n} \int_{\mathbb{R}^n} e^{i x \cdot \xi} \langle \xi \rangle^{-2s} \hat{f}(\xi) \, d\xi, \quad f \in B^0.
\]
It is clear that $K_s$ belongs to the space of bounded linear operators $\mathscr{B}(B^0, B^{2s})$, and its operator norm satisfies $\|K_s\|_{\mathscr{B}(B^0, B^{2s})} \leq (2\pi)^{-n}$. Moreover, we can establish the following:
\begin{equation}\label{5.1}
(1 - \Delta)^{-s} f(x) = (2\pi)^{-n} \int_{\mathbb{R}^n} e^{i x \cdot \xi} \langle \xi \rangle^{-2s} \hat{f}(\xi) \, d\xi, \quad f \in B^0.
\end{equation}
In other words,  we have verified that $K_s$ is indeed the inverse of the operator $(1 - \Delta)^s$.

\begin{remark}\label{rem1.0}
{\rm
Let $s \geq 0$, $t \geq 0$, and let $a: \mathbb{R}^n \rightarrow \mathbb{C}$ be a measurable function such that $\langle \xi \rangle^{s-t} a \in L^\infty$. Consider the pseudo-differential operator defined by
\[
P f(x) = (2\pi)^{-n} \int_{\mathbb{R}^n} e^{i x \cdot \xi} a(\xi) \hat{f}(\xi) \, d\xi.
\]
We can verify that $P$ belongs to the space of bounded linear operators $\mathscr{B}(B^t, B^s)$, and its operator norm satisfies $\|P\|_{\mathscr{B}(B^t, B^s)} \leq \|\langle \xi \rangle^{s-t} a\|_\infty$.
In particular, if $a(\xi) = \sum_{|\alpha| \leq k} a_\alpha \xi^\alpha$, where $a_\alpha \in \mathbb{C}$ and $k\in\mathbb{N}$, then $P$ reduces to the constant coefficients differential operator
\[
P = \sum_{|\alpha| \leq k} (-i)^{|\alpha|} a_\alpha \partial^\alpha.
\]
Therefore, $P \in \mathscr{B}(B^t, B^s)$ provided that $s + k \le t $.
}
\end{remark}

Other notable properties of the spaces $B^s$ are summarized in the following proposition where the last item (iv) can also be found in \cite[Lemma 3.4]{CLLZ}.

\begin{proposition}\label{pro2}
$\mathrm{(i)}$ Let $0 \leq r \leq t$, $s \in [r, t]$, and $\alpha \in [0, 1]$ such that $s = \alpha r + (1 - \alpha) t$. Then, the following inequality holds:
\begin{equation}\label{ii}
\|f\|_{B^s} \leq \|f\|_{B^r}^\alpha \|f\|_{B^t}^{1 - \alpha}, \quad f \in B^t.
\end{equation}
$\mathrm{(ii)}$ The Schwartz space $\mathscr{S}$ is dense in $B^s$ for all $s \geq 0$.
\\
$\mathrm{(iii)}$ Let $s \geq 0$ and $t > s + \frac{n}{2}$. Then, $\langle \xi \rangle^{s - t} \in L^2$, and the Sobolev space $H^t$ is continuously embedded in $B^s$, denoted by $H^t \hookrightarrow B^s$, with the following estimate:
\[
\|f\|_{B^s} \leq \|\langle \xi \rangle^{s - t}\|_2 \|f\|_{H^t}, \quad f \in H^t.
\]
$\mathrm{(iv)}$ Let $s \geq 0$. If $f, g \in B^s$, then the product $fg$ also belongs to $B^s$, and the following inequality holds:
\begin{equation}\label{prod}
\|fg\|_{B^s} \leq 2^{s/2} (2\pi)^{-n} \|g\|_{B^s} \|f\|_{B^s}.
\end{equation}
\end{proposition}

\begin{proof}
(i) Let $f \in B^t$. We start by noting that
\[
\int_{\mathbb{R}^n} \langle \xi \rangle^s |\hat{f}(\xi)| \, d\xi = \int_{\mathbb{R}^n} \left[ \langle \xi \rangle^{\alpha r} |\hat{f}(\xi)|^\alpha \right] \left[ \langle \xi \rangle^{(1 - \alpha) t} |\hat{f}(\xi)|^{1 - \alpha} \right] \, d\xi.
\]
Applying H\"older's inequality, we obtain the desired result \eqref{ii}.

(ii) Let $s \geq 0$ and $f \in B^s$. For all $\epsilon > 0$, using the density of the Schwartz space $\mathscr{S}$ in $L^1$, we can find $\psi \in \mathscr{S}$ such that
\[
\|\langle \xi \rangle^s \hat{f} - \psi\|_1 \leq \epsilon.
\]
Then, define $\varphi := \mathscr{F}^{-1}(\langle \xi \rangle^{-s} \psi) \in \mathscr{S}$. We have
\[
\|f - \varphi\|_{B^s} = \|\langle \xi \rangle^s \hat{f} - \psi\|_1 \leq \epsilon.
\]
This shows that $\mathscr{S}$ is dense in $B^s$.

(iii) The fact that $\langle \xi \rangle^{s - t} \in L^2$ is straightforward. Let $f \in H^t$. Using the Cauchy-Schwarz inequality, we get
\[
\|\langle \xi \rangle^s \hat{f}\|_1 = \|\langle \xi \rangle^{s - t} \langle \xi \rangle^t \hat{f}\|_1 \leq \|\langle \xi \rangle^{s - t}\|_2 \|\langle \xi \rangle^t \hat{f}\|_2 = \|\langle \xi \rangle^{s - t}\|_2 \|f\|_{H^t}.
\]
The expected embedding $H^t \hookrightarrow B^s$ then follows.

(iv) For $f, g \in B^s$, we have
\[
|\widehat{fg}(\xi)| \leq (2\pi)^{-n} \int_{\mathbb{R}^n} |\hat{f}(\xi - \eta)| |\hat{g}(\eta)| \, d\eta.
\]
This inequality, combined with Peetre's inequality:
\[
\langle \xi \rangle^s \leq 2^{s/2} \langle \xi - \eta \rangle^s \langle \eta \rangle^s, \quad \xi, \eta \in \mathbb{R}^n,
\]
implies
\[
\begin{aligned}
\int_{\mathbb{R}^n} \langle \xi \rangle^s |\widehat{fg}(\xi)| \, d\xi
&\leq 2^{s/2} (2\pi)^{-n} \int_{\mathbb{R}^n} \int_{\mathbb{R}^n} \langle \xi - \eta \rangle^s |\hat{f}(\xi - \eta)| \langle \eta \rangle^s |\hat{g}(\eta)| \, d\eta \, d\xi \\
&\leq 2^{s/2} (2\pi)^{-n} \int_{\mathbb{R}^n} \langle \xi \rangle^s |\hat{f}(\xi)| \, d\xi \int_{\mathbb{R}^n} \langle \eta \rangle^s |\hat{g}(\eta)| \, d\eta \\
&\leq 2^{s/2} (2\pi)^{-n} \|f\|_{B^s} \|g\|_{B^s}.
\end{aligned}
\]
Thus, the inequality \eqref{prod} follows.
\end{proof}

For any compact set $K \subset \mathbb{R}^n$ and $s \geq 0$, we define $B_K^s:=B^s\cap C_K^0$, where
\[
C_K^0 := \{f \in C^0; \; \mathrm{supp}\, f \subset K\}.
\]
Here, $C^0$ denotes the space of continuous functions on $\mathbb{R}^n$. As a closed subspace of $B^s$, $B_K^s$ is a Banach space when equipped with the norm $\|\cdot\|_{B^s}$.

We now present the following compact embedding theorem, whose proof is inspired by that of \cite[Theorem 10.1.10]{Ho}.
\begin{theorem}\label{thmcem}
For all $0 \leq s < t$, the embedding $B_K^t \hookrightarrow B^s$ is compact.
\end{theorem}

\begin{proof}
Let $(f_j)$ be a sequence in $B_K^t$ such that $\sup_j \|f_j\|_{B^t} \leq 1$. Fix $\phi \in C_0^\infty$ satisfying $\phi = 1$ in a neighborhood of $K$. In this case, since
\[
\hat{f}_j = \widehat{\phi f}_j = (2\pi)^{-n} \hat{f}_j \ast \hat{\phi}, \quad j \geq 1,
\]
we have, for all $\alpha \in (\mathbb{N}\cup\{0\})^n$ and $j \geq 1$
\[
\partial^\alpha \hat{f}_j = (2\pi)^{-n} \hat{f}_j \ast \partial^\alpha \hat{\phi}.
\]
Consequently,
\[
|\partial^\alpha \hat{f}_j(\xi)| \leq (2\pi)^{-n} \|f_j\|_{B^0} \|\partial^\alpha \hat{\phi}\|_\infty \leq (2\pi)^{-n} \|\partial^\alpha \hat{\phi}\|_\infty, \quad \xi \in \mathbb{R}^n, \; j \geq 1.
\]
In particular, the sequence $(\hat{f}_j)$ is uniformly bounded and equicontinuous on all compact subsets of $\mathbb{R}^n$.

For all $\epsilon > 0$, we can find $\rho > 0$ such that $\langle \xi \rangle^{-(t - s)} \leq \epsilon$ if $\langle \xi \rangle > \rho$. Without loss of generality, by passing to a subsequence if necessary, we may assume that $(\hat{f}_j)$ converges in $C^0(\{\langle \xi \rangle \leq \rho\})$. Using the following estimate
\[
\begin{aligned}
\|f_j - f_k\|_{B^s} &= \int_{\langle \xi \rangle \leq \rho} \langle \xi \rangle^s |\hat{f}_j(\xi) - \hat{f}_k(\xi)| \, d\xi + \int_{\langle \xi \rangle > \rho} \langle \xi \rangle^s |\hat{f}_j(\xi) - \hat{f}_k(\xi)| \, d\xi \\
&\leq |\{\langle \xi \rangle \leq \rho\}| \|\hat{f}_j - \hat{f}_k\|_{C^0(\{\langle \xi \rangle \leq \rho\})} + \epsilon \|f_j - f_k\|_{B^t} \\
&\leq |\{\langle \xi \rangle \leq \rho\}| \|\hat{f}_j - \hat{f}_k\|_{C^0(\{\langle \xi \rangle \leq \rho\})} + 2\epsilon, \quad j \geq 1, \; k \geq 1,
\end{aligned}
\]
we find that $(f_j)$ is a Cauchy sequence in $B^s$. Since $B^s$ is complete, $(f_j)$ converges in $B^s$. This completes the proof.
\end{proof}

We now demonstrate that the convolution with a function from $B_K^0$ defines a bounded operator on $B^s$ for all $s \geq 0$. This follows from the embedding $B_K^0 \hookrightarrow L^1$, given by
\[
\|f\|_1 \leq |K| \|f\|_\infty \leq (2\pi)^{-n} |K| \|\hat{f}\|_1 = (2\pi)^{-n} |K| \|f\|_{B^0}, \quad f \in B_K^0,
\]
and the following lemma.
\begin{lemma}\label{lemconv}
Let $g \in L^1$ and $f \in B^s$, where $s \geq 0$. Then $f \ast g \in B^s$ and
\[
\|f \ast g\|_{B^s} \leq \|f\|_{B^s} \|g\|_1.
\]
\end{lemma}
\begin{proof}
Since $\|\hat{g}\|_\infty \leq \|g\|_1$, we have
\[
\langle \xi \rangle^s |\widehat{f \ast g}| = \langle \xi \rangle^s |\hat{f}| |\hat{g}| \leq \|g\|_1 \langle \xi \rangle^s |\hat{f}|,\quad \xi\in\mathbb{R}^n
\]
as expected.
\end{proof}

The remaining portion of this subsection will be dedicated to the rigorous definition of the dual space of the spectral Barron space $B^s$ for $s \geq 0$. 
We define, for $s\geq 0$, the space $\tilde{B}^{-s}$ as follows:
\[
\tilde{B}^{-s}:=\left\{ f\in \mathscr{S}';\; \langle \xi\rangle^{-s}\mathscr{F}^{-1}f\in L^\infty\right\}.
\]
We proceed to verify that $\tilde{B}^{-s}$ is a Banach space when equipped with the norm
\[
\|f\|_{\tilde{B}^{-s}}:=\left\|\langle \xi\rangle^{-s}\mathscr{F}^{-1}f\right\|_{\infty}.
\]
For all $\varphi \in \mathscr{S}$, $f\in \mathscr{S}'$, and $s\ge 0$, the following pairing identity holds:
\begin{equation}\label{df}
\langle f,\varphi\rangle = \left\langle \langle \xi\rangle^{-s}\mathscr{F}^{-1}f, \langle \xi\rangle^s\hat{\varphi}\right\rangle.
\end{equation}
In particular, when $\varphi \in \mathscr{S}$ and $f\in \tilde{B}^{-s}$, we can express the pairing as an integral
\[
\langle f,\varphi\rangle = \left\langle \langle \xi\rangle^{-s}\mathscr{F}^{-1}f, \langle \xi\rangle^s\hat{\varphi}\right\rangle:=\int_{\mathbb{R}^n}\left[\langle \xi\rangle^{-s}\mathscr{F}^{-1}f\right]\left[\langle \xi\rangle^s\hat{\varphi}\right]d\xi.
\]
Given the last identity, and considering that the Schwartz space $\mathscr{S}$ is dense in the spectral Barron space $B^s$, we can claim that every $f\in\tilde{B}^{-s}$ has a unique bounded linear extension to $(B^s)'$, which denotes the dual space of $B^s$.  We still denote this extension by $f$, and it satisfies the inequality
\[
\|f\|_{(B^s)'}\leq \|f\|_{\tilde{B}^{-s}}.
\]
In other words, we have the continuous embedding $\tilde{B}^{-s}\hookrightarrow (B^s)'$.
In fact, the reverse inclusion also holds, as demonstrated by the following result.

\begin{proposition}\label{pro3}
For all $s\geq 0$, the dual space $(B^s)'$ of the spectral Barron space $B^s$ can be identified both algebraically and topologically with the space $\tilde{B}^{-s}$.
\end{proposition}

\begin{proof}
Let $f\in (B^s)'\subseteq \mathscr{S}'$. Then, for all $\varphi\in \mathscr{S}$, utilizing the pairing identity \eqref{df}, we have
\[
\left|\left\langle \langle \xi\rangle^{-s}\mathscr{F}^{-1}f, \langle \xi\rangle^s\hat{\varphi}\right\rangle\right|\leq \|f\|_{(B^s)'}\|\varphi \|_{B^s}.
\]
We know that the mapping $\varphi \in \mathscr{S}\mapsto \langle \xi\rangle^s\hat{\varphi}\in \mathscr{S}$ is an isomorphism. This implies that 
\[
\left|\left\langle \langle \xi\rangle^{-s}\mathscr{F}^{-1}f, \psi\right\rangle\right|\leq \|f\|_{(B^s)'}\|\psi \|_1,\quad \psi\in \mathscr{S}.
\]

Since $\mathscr{S}$ is dense in $L^1$, $f$ admits a unique continuous linear extension, still denoted by $f$, to the space $\tilde{B}^{-s}$ and
\[
\| f\|_{\tilde{B}^{-s}} \leq \|f\|_{(B^s)'}.
\]
\end{proof}

\subsection{Embedding of spectral Barron spaces}
In this subsection, we discuss the  embedding properties of spectral Barron spaces into various functional spaces.

\subsubsection{Spectral Barron spaces and interpolation}

For $ 0 < \theta < 1 $, the real interpolation space $ (B^r, B^t)_{\theta, 1} $ is defined as follows:
\[
(B^r, B^t)_{\theta, 1} := \left\{ f \in B^r + B^t; \; \rho^{-\theta} K(\cdot, f) \in L^1\left((0, \infty), d\rho/\rho \right) \right\},
\]
where the functional $ K(\rho, f) $ is given by
\[
K(\rho, f) := \inf \left\{ \|g\|_{B^r} + \rho \|h\|_{B^t}; \; g \in B^r, \; h \in B^t \; \text{such that} \; g + h = f \right\}.
\]
The Banach space $ (B^r, B^t)_{\theta, 1} $ is typically endowed with its natural norm
\[
\|f\|_{(B^r, B^t)_{\theta, 1}} := \left\| \rho^{-\theta} K(\cdot, f) \right\|_{L^1\left((0, \infty), \frac{d\rho}{\rho}\right)}.
\]

For further use, it is useful to note  that for all $0\le r\le s\le t$ we have
\[
B^t=B^t\cap B^r\hookrightarrow B^s \hookrightarrow B^r+B^t=B^r.
\]
This follows from the fact that $B^t\hookrightarrow B^s$ whenever $0\le s\le t$.

\begin{theorem}\label{thmemb}
Let $ 0 \leq r < s < t $, $ \theta = \frac{s - r}{t - r} $, and $ 0 < \tilde{\theta} < \theta $. Then we have the following continuous embeddings:
\[
(B^r, B^t)_{\theta, 1} \hookrightarrow B^s \hookrightarrow (B^r, B^t)_{\tilde{\theta}, 1}.
\]
\end{theorem}

\begin{proof}
By \cite[Definition 1.19]{Lu} and  \eqref{ii}, the space $ B^s $ belongs to the class $ J_\theta(B^r, B^t) $. Combining this with Proposition \ref{pro2} and \cite[Proposition 1.20]{Lu}, we deduce that $ (B^r, B^t)_{\theta, 1} \hookrightarrow B^s $.

To prove the second embedding, we start by decomposing each $ f \in B^s $ into two terms:
\[
f = g_\mu + h_\mu,\quad \mu >1,
\]
where
\[
g_\mu := \mathscr{F}^{-1}(\chi_{\{\langle \xi \rangle > \mu\}} \hat{f}), \quad h_\mu := \mathscr{F}^{-1}(\chi_{\{\langle \xi \rangle \leq \mu\}} \hat{f}).
\]
We verify that $ g_\mu \in B^r $ and $ h_\mu \in B^t $:
\[
\int_{\mathbb{R}^n} \langle \xi \rangle^r |\hat{g}_\mu(\xi)| \, d\xi = \int_{\{\langle \xi \rangle > \mu\}} \langle \xi \rangle^r |\hat{f}(\xi)| \, d\xi \leq \mu^{-(s - r)} \int_{\mathbb{R}^n} \langle \xi \rangle^s |\hat{f}(\xi)| \, d\xi,
\]
and
\[
\int_{\mathbb{R}^n} \langle \xi \rangle^t |\hat{h}_\mu(\xi)| \, d\xi = \int_{\{\langle \xi \rangle \leq \mu\}} \langle \xi \rangle^t |\hat{f}(\xi)| \, d\xi \leq \mu^{(t - s)} \int_{\mathbb{R}^n} \langle \xi \rangle^s |\hat{f}(\xi)| \, d\xi.
\]
Thus, we have
\[
\|g_\mu\|_{B^r} \leq \mu^{-(s - r)} \|f\|_{B^s}, \quad \|h_\mu\|_{B^t} \leq \mu^{(t - s)} \|f\|_{B^s}.
\]
For all $ \mu > 1 $, the $ K $-functional satisfies
\[
K(\rho, f) \leq \left( \mu^{-(s - r)} + \rho \mu^{(t - s)} \right) \|f\|_{B^s}, \quad \rho > 0.
\]

In the case $ 0 < \rho < 1 $, we choose $ \mu = \rho^{-1/(t - r)} $ in the inequality above to obtain
\[
K(\rho, f) \leq 2 \rho^{(s - r)/(t - r)} \|f\|_{B^s} = 2 \rho^\theta \|f\|_{B^s}, \quad 0 < \rho < 1.
\]
Hence, $ \rho^{-\tilde{\theta}} K(\rho, f) \in L^1((0, 1), d\rho/\rho) $. On the other hand, if $ \rho \geq 1 $, since $ K(\rho, f) \leq \|f\|_{B^r} \leq \|f\|_{B^s} $ for all $ f \in B^s $, we have $ \rho^{-\tilde{\theta}} K(\rho, f) \in L^1((1, \infty), d\rho/\rho) $.

Combining these two cases, we conclude that $ \rho^{-\tilde{\theta}} K(\rho, f) \in L^1((0, \infty), d\rho/\rho) $, and thus
\[
\|f\|_{(B^r, B^t)_{\tilde{\theta}, 1}} = \left\| \rho^{-\tilde{\theta}} K(\rho, f) \right\|_{L^1((0, \infty), d\rho/\rho)} \leq \mathbf{c} \|f\|_{B^s}, \quad f \in B^s,
\]
where $ \mathbf{c} = \mathbf{c}(r, s, t) > 0 $ is a constant. Therefore, $ B^s \hookrightarrow (B^r, B^t)_{\tilde{\theta}, 1} $.
\end{proof}
The above theorem demonstrates that $B^s$ is nearly an interpolation space; however, it remains unknown whether $B^s$ coincides exactly with $(B^r, B^t)_{\theta,1}$.

\subsubsection{Embedding by Sobolev spaces}

For the case $ n \geq 2 $, we invoke \cite[Theorem 3.4]{PW} to establish the following Hardy-type inequality:
\begin{equation}\label{har1}
\int_{\mathbb{R}^n}(1+|\xi|)^{1-n}|\hat{f}|d\xi\le \vartheta \|\nabla f\|_1,\quad f\in W^{1,1},
\end{equation}
where $ \vartheta > 0 $ is a constant independent of $ f $.

The above observation yields the following results. 

\begin{theorem}\label{thm_emSobolev}
Let $k \geq 0$ be an integer. Then, $W^{n+k,1} \hookrightarrow B^k$.
\end{theorem}

\begin{proof}
Let $ m \geq 0 $ be an integer and $ f \in W^{m+1,1}$. By applying the above inequality to each partial derivative $ \partial^\alpha f $ for all multi-indices $ \alpha $ with $ |\alpha| \leq m $, we obtain
\begin{equation}\label{har2}
\int_{\mathbb{R}^n}|\xi|^m(1+|\xi|)^{1-n}|\hat{f}|d\xi\le \mathbf{c} \|f\|_{W^{m+1,1}},
\end{equation}
where $ \mathbf{c} = \mathbf{c}(n, m)>0 $ denotes a generic constant that may depend on $ n $ and $ m $.

Combining the inequalities \eqref{har1} and \eqref{har2}, we have
\[
\int_{\mathbb{R}^n} (1 + |\xi|^m) (1 + |\xi|)^{1-n} |\hat{f}(\xi)| \, d\xi \leq \mathbf{c} \|f\|_{W^{m+1,1}}.
\]
Furthermore, observe that
\[
\frac{1 + |\xi|^m}{(1 + |\xi|)^m} \geq 2^{1 - m}, \quad  |\xi| \geq 0.
\]
This inequality implies
\[
(1 + |\xi|^m) (1 + |\xi|)^{1-n} \geq 2^{1 - m} (1 + |\xi|)^{m + 1 - n}.
\]
Therefore, we deduce
\[
\int_{\mathbb{R}^n} \langle \xi \rangle^{m + 1 - n} |\hat{f}(\xi)| \, d\xi \leq \mathbf{c} \|f\|_{W^{m+1,1}}.
\]
Let $ k \geq 0 $ be an integer. By choosing $ m $ such that $ m + 1 - n = k $, we derive
\[
\int_{\mathbb{R}^n} \langle \xi \rangle^k |\hat{f}(\xi)| \, d\xi \leq \mathfrak{c} \|f\|_{W^{n + k, 1}}, \quad f \in W^{n + k, 1},
\]
where $ \mathfrak{c} = \mathfrak{c}(n, k) > 0 $ is a constant. That is we have proved that $W^{n+k,1}\hookrightarrow B^k$. 
\end{proof}

\subsubsection{Embedding into H\"older spaces}

We conduct a further in-depth investigation into the embedding properties of spectral Barron spaces into H\"older spaces. 
Before presenting the main result, we make several useful observations. Let $k \geq 0$ be an integer. From the relation $\mathscr{F}((-i)^{|\alpha|} \partial^\alpha f) = \xi^\alpha \hat{f}$, where $f \in \mathscr{S}'$, and by iterating Lemma \ref{lemB1}, we deduce that $B^k \hookrightarrow C_b^k$,  and therefore $B^k\hookrightarrow W^{k,\infty}$.

\begin{remark}\label{remess}
{\rm
Since $B^k\hookrightarrow C_b^k$ for all integer $k\ge 0$, Theorem \ref{thm_emSobolev} yields the known result $W^{n+k,1}\hookrightarrow C_b^k$ (e.g., \cite[Theorem 4.12, (1)]{AF}).
}
\end{remark}

Now, let $s = k + \theta$, where $k \geq 0$ is an integer and $0 < \theta < 1$. Since $B^s \hookrightarrow B^k$, it follows that $B^s \hookrightarrow C_b^k$. Fix a multi-index $\alpha = (\alpha_1, \ldots, \alpha_n) \in (\mathbb{N}\cup\{0\})^n$ such that $|\alpha| = k$, and set $g = \partial^\alpha f$. We can verify that $g \in B^\theta$.

For $0 < \beta \leq 1$, we define the semi-norm $[\cdot]_\beta$ as follows:
\[
[f]_\beta := \sup \left\{ \frac{|f(x) - f(y)|}{|x - y|^\beta}; \; x, y \in \mathbb{R}^n, x \neq y \right\}.
\]
We then define the space $C_b^{k,\beta}$ as
\[
C_b^{k,\beta} := \left\{ f \in C_b^k ; \; [\partial^\alpha f]_\beta < \infty, \; |\alpha| = k \right\}.
\]
It is a classical result that $C_b^{k,\beta}$ is a Banach space with respect to the norm
\[
\|f\|_{C_b^{k,\beta}} := \|f\|_{C_b^k} + \sup_{|\alpha| = k} [\partial^\alpha f]_\beta.
\]

\begin{theorem}\label{thmehs}
Let $k \geq 0$ be an integer and $0 < \theta < 1$. Then, $B^{k+\theta} \hookrightarrow C_b^{k,\gamma}$ for all $0 < \gamma < \theta$.
\end{theorem}

\begin{proof}
In light of the preceding observations, it suffices to provide the proof for the case $k = 0$. Let $0 < \gamma < \beta < \theta$. From Theorem \ref{thmemb}, we have $B^\theta \hookrightarrow (B^0, B^1)_{\beta, 1}$, and by \cite[Proposition 1.4]{Lu}, $(B^0, B^1)_{\beta, 1} \hookrightarrow (B^0, B^1)_{\gamma, \infty}$. Here, the space $(B^0, B^1)_{\gamma, \infty}$ is defined as
\[
(B^0, B^1)_{\gamma, \infty} := \left\{ f \in B^0 + B^1 ; \; \rho^{-\gamma} K(\cdot, f) \in L^\infty((0, \infty)) \right\},
\]
equipped with the norm
\[
\|f\|_{(B^0, B^1)_{\gamma, \infty}} := \|\rho^{-\gamma} K(\cdot, f)\|_{L^\infty((0, \infty))}.
\]

Let $f \in B^\theta \subset B^0 + B^1$ with $f = g + h$, where $g \in B^0$ and $h \in B^1$. In this case, we have
\begin{equation}\label{ho1}
\|g\|_{C_b^0} \leq c \|g\|_{B^0}, \quad \|h\|_{C_b^1} \leq c \|h\|_{B^1},
\end{equation}
where $c = c(\theta) > 0$ is a constant.

Let $x, y \in \mathbb{R}^n$ with $x \neq y$. Using \eqref{ho1}, we obtain
\[
|f(x) - f(y)| \leq 2c \|g\|_{C_b^0} + \|h\|_{C_b^1} |x - y| \leq 2c \|g\|_{B^0} + c \|h\|_{B^1} |x - y|.
\]
Taking the infimum with respect to $g$ and $h$ satisfying $g + h = f$, we obtain
\[
|f(x) - f(y)| \leq 2c K(|x - y|, f) \leq 2c |x - y|^\gamma \|f\|_{(B^0, B^1)_{\gamma, \infty}}.
\]
This, combined with the fact that $B^\theta \hookrightarrow (B^0, B^1)_{\gamma, \infty}$, implies
\begin{equation}\label{hol}
|f(x) - f(y)| \leq c' |x - y|^\gamma \|f\|_{B^\theta},
\end{equation}
where $c' = c'(\theta, \gamma)>0$ is a constant. The proof is complete.
\end{proof}

\begin{remark}\label{remehs}{\rm
We observe that \eqref{hol} shows that the unit ball of $B^\theta$ is equicontinuous.
}\end{remark}

As a consequence of Theorem \ref{thmehs}, we have the following embedding result, where we set
\begin{align*}
&W_{\rm{loc}}^{k+\sigma,p}:=\{f:\mathbb{R}^n\rightarrow \mathbb{C}\; \mbox{measurable};
\\
&\hskip 4cm f|_{\Omega}\in W^{k+\sigma,p}(\Omega)\; \mbox{for all bounded subset}\; \Omega\; \mbox{of}\; \mathbb{R}^n\}.
\end{align*}
\begin{corollary}\label{CR}
For all $p\ge 1$, an integer $k\ge 0$, and for all $0<\theta<1$ and $0<\sigma <\frac{\theta}{p}$, $B^{k+\theta}$ is embedded in $W_{\rm{loc}}^{k+\sigma,p}$.
\end{corollary}
\begin{proof}
Let $k\ge 0$ be an integer, $0<\theta<1$, and $0<\gamma <\theta$. By Theorem \ref{thmehs}, we have the embedding $B^{k+\theta}\hookrightarrow C_b^{k,\gamma}$. On the other hand, for any bounded open subset $\Omega \subset \mathbb{R}^n$, it follows from \cite[Definition 1.3.2.1]{Gr} that $C^{k,\gamma}(\overline{\Omega})\hookrightarrow W^{k+\sigma,p}(\Omega)$ whenever $p\ge 1$ and $\sigma p < \gamma$. This completes the proof.
\end{proof}

In light of \cite[Theorem 8.25]{Ch}, we obtain the following result as a consequence of Theorem \ref{thmehs}.

\begin{corollary}\label{corhm}
Let $k \geq 0$ be an integer and $0 < \gamma < \theta < 1$. For all $m \in B^{k+\theta}$, the operator $f \in H^{k+\gamma} \mapsto mf \in H^{k+\gamma}$ is bounded.
\end{corollary}

The above corollary then inspires us to investigate the following compact multiplier results, whose proof was inspired by that of \cite[Theorem 2.3]{CLLZ}. 
\begin{proposition}\label{procm}
Let $ 0 \leq s < t $ and fix $ m \in B^t $. Then the multiplication operator
\[
\mathbf{M}_m: f\in B^t\mapsto mf\in B^s
\]
is compact.
\end{proposition}

\begin{proof}
We aim to demonstrate that the closure of the set  
\[
\mathcal{C} := \left\{ g = \langle \xi \rangle^s \widehat{mf} ; \; \langle \xi \rangle^t \hat{f} \in L^1, \; \|\langle \xi \rangle^t \hat{f}\|_1 \leq 1 \right\}  
\]  
is compact in $ L^1 $.  

Given that $ \widehat{mf} = (2\pi)^{-n} \hat{m} \ast \hat{f} $, it suffices to show that the set  
\[
\mathcal{L} := \left\{ g = \langle \xi \rangle^s h \ast f ; \; \langle \xi \rangle^t f \in L^1, \; \|\langle \xi \rangle^t f\|_1 \leq 1 \right\},  
\]  
where $ h = \hat{m} \in L^1 $ satisfies $ \langle \xi \rangle^t h \in L^1 $, has a compact closure in $ L^1 $.  

Define $ \vartheta := 2^{t/2} \|\langle \xi \rangle^t h\|_1 = 2^{t/2} \|m\|_{B^t} $. We utilize the following inequality, whose proof is analogous to that of (iv) in Proposition \ref{pro2}:  
\[
\|\langle \xi \rangle^t h \ast f\|_1 \leq 2^{t/2} \|\langle \xi \rangle^t h\|_1 \|\langle \xi \rangle^t f\|_1 \leq \vartheta \|\langle \xi \rangle^t f\|_1 \quad \text{if} \; \langle \xi \rangle^t f \in L^1.  
\]  
Let $ g = \langle \xi \rangle^s h \ast f \in \mathcal{L} $. Then,  
\[
\int_{|\xi| > \rho} |g(\xi)| \, d\xi \leq \sup_{|\xi| > \rho} \langle \xi \rangle^{-(t-s)} \|\langle \xi \rangle^t h \ast f\|_1 \leq \vartheta \sup_{|\xi| > \rho} \langle \xi \rangle^{-(t-s)}.  
\]  
Consequently, for any $ \epsilon > 0 $, there exists $ \rho_\epsilon = \rho_\epsilon(\epsilon, s, t) $ such that for all $ \rho \geq \rho_\epsilon $ and $ |\eta| \leq 1 $,  
\begin{equation}\label{ce1}
\int_{|\xi|>\rho}|g(\xi+\eta)|d\xi \le \epsilon.
\end{equation}

For $ \epsilon > 0 $, $ |\eta| < 1 $, and $ \rho \geq \rho_\epsilon $, we decompose  
\[
\int_{\mathbb{R}^n} \left| g(\xi + \eta) - g(\xi) \right| d\xi = I_1 + I_2,  
\]  
where  
\begin{align*}
&I_1 := \int_{|\xi| > \rho} |g(\xi + \eta) - g(\xi)| \, d\xi, 
\\
&I_2 := \int_{|\xi| \leq \rho} |g(\xi + \eta) -g(\xi) |\, d\xi.  
\end{align*} 
From the preceding inequality  \eqref{ce1}, it follows that
\begin{equation}\label{ce2}
|I_1| \le 2\epsilon .
\end{equation}

Next, we note that  
\[
I_2 \le J_1 + J_2,  
\]  
where  
\[
J_1 := \int_{|\xi| \leq \rho} | \left( \langle \xi + \eta \rangle^s - \langle \xi \rangle^s \right) h \ast f(\xi + \eta)| \, d\xi,  
\]  
\[
J_2 := \int_{|\xi| \leq \rho}| \langle \xi \rangle^s \left( h \ast f(\xi + \eta) - h \ast f(\xi) \right)| \, d\xi.  
\]  
Since $ \xi \mapsto \langle \xi \rangle^s $ is uniformly continuous on the ball $ B(0, \rho + 1) $, there exists $ \delta_0 = \delta_0(\epsilon, s)>0 $ such that  
\[
\left| \langle \xi + \eta \rangle^s - \langle \xi \rangle^s \right| \leq \vartheta^{-1} \epsilon, \quad |\xi| \leq \rho, \; |\eta| \leq \delta_0.  
\]  
This implies  
\[
|J_1| \leq \epsilon, \quad  |\eta| \leq \delta_0.  
\]  
Let $ \upsilon = \max_{|\xi| \leq \rho} \langle \xi \rangle^s $. Then,  
\begin{equation}\label{ce4}
|J_2|\le \upsilon \int_{|\xi|\le \rho} |h\ast f(\xi+\eta)-h\ast f(\xi)|d\xi=:\upsilon \tilde{J}_2.
\end{equation}
Choose $ \phi \in C_0^\infty $ such that $ \|\phi - h\|_1 \leq \epsilon $, and let  
\begin{equation}\label{ce4.1}
\tilde{J}_2=K_1+K_2,
\end{equation}
where  
\begin{align*}
& K_1 := \int_{|\xi| \leq \rho} \left| (h - \phi) \ast f(\xi + \eta) - (h - \phi) \ast f(\xi) \right| \, d\xi,  \\
& K_2 := \int_{|\xi| \leq \rho} \left| \phi \ast f(\xi + \eta) - \phi \ast f(\xi) \right| \, d\xi.  
\end{align*}
We observe that  
\begin{equation}\label{ce5}
K_1\le 2\|h-\phi\|_1\|f\|_1\le 2\|h-\phi\|_1\le 2\epsilon.
\end{equation}
On the other hand,  
\begin{align*}
K_2 & = \int_{|\xi| \leq \rho} \left| \int_{\mathbb{R}^n} \phi(y) \left( f(\xi + \eta - y) - f(\xi - y) \right) dy \right| \, d\xi   \\
& \leq \int_{|\xi| \leq \rho} \int_{\mathbb{R}^n} |\phi(\xi + \eta - y) - \phi(\xi - y)| |f(y)| \, dy \, d\xi \\
& \leq \int_{\mathbb{R}^n}|f(y)| \int_{|\xi|\le \rho} |\phi(\xi+\eta-y) -\phi(\xi-y)| d\xi dy.
\end{align*}
Since $ \phi $ has compact support, it is uniformly continuous. Thus, there exists $ \delta_1 = \delta_1(\epsilon, h) >0 $ such that  
\[
|\phi(\xi + \eta - y) - \phi(\xi - y)| \leq |B(0, \rho)|^{-1} \epsilon, \quad  y \in \mathbb{R}^n, \; |\xi| \leq \rho, \; |\eta| \leq \delta_1.  
\]  
This yields  
\begin{equation}\label{ce6}
K_2 \leq \epsilon, \quad  |\eta| \leq \delta_1.  
\end{equation}
Let $ \delta := \min(\delta_0, \delta_1) $. Combining \eqref{ce2}, \eqref{ce4}, \eqref{ce4.1}, \eqref{ce5}, and \eqref{ce6}, we obtain  
\begin{equation}\label{ce7}
\int_{\mathbb{R}^n}|g(\xi+\eta)-g(\xi)|d\xi\le (2+3\upsilon)\epsilon, \quad |\eta|\le \delta.
\end{equation}
In light of \eqref{ce1} and \eqref{ce7}, by applying \cite[Corollary 4.27]{Br}, we deduce that $ \mathcal{L} $ has a compact closure in $ L^1 $. This completes the proof.  
\end{proof}

Let $s > 0$ and $m \in B^0$. By leveraging formula \eqref{5.1}, we can establish the following identity:
\[
\mathscr{F}((1 - \Delta)^{-s}(mf)) = (2\pi)^{-n}\langle \xi \rangle^{-2s} \widehat{mf},\quad f \in B^0.
\]

With this identity in hand, we can follow a similar line of reasoning as that used in the proof of Proposition \ref{procm} to derive the following result.
\begin{proposition}\label{procm2}
Let $s > 0$ and $m \in B^0$. Then the operator $(1 - \Delta)^{-s}m: B^0 \rightarrow B^0$ is compact.
\end{proposition}
It is worth noting that Proposition \ref{procm2} is already contained in \cite[Proposition 3.5]{CLLZ}.

\subsection{Variants of spectral Barron spaces} 
Recalling the definition of spectral Barron spaces and the associated norm as defined in \eqref{eq_SBarronSpace} and \eqref{eq_SBarronSpacenorm}, we present an extended discussion of alternative variants of spectral Barron spaces appearing in other references. 

\subsubsection{Spectral Barron spaces in $L^p$.}  
Let $ 1 \le p \le 2 $ and let $ 2 \le p' \le \infty $ denote its Hölder conjugate, i.e., $ \frac{1}{p} + \frac{1}{p'} = 1 $. By the Hausdorff-Young theorem (e.g., \cite[Theorem 8.10]{Ch}), the Fourier transform $ \mathscr{F} $ maps $ L^p $ continuously to $ L^{p'} $ with  
\[
\|\hat{f}\|_{p'} \le (2\pi)^{n/p'} \|f\|_p, \quad f \in L^p.
\]  
For $ s \in \mathbb{R} $, define the spectral Barron space $ \mathcal{B}^{p,s} $ as  
\[
\mathcal{B}^{p,s} := \left\{ f \in L^p ; \; |\xi|^s \hat{f} \in L^1 \right\}.
\]  
Equipped with the norm
\begin{equation}\label{inn}
\|f\|_{\mathcal{B}^{p,s}} := \|f\|_p + \big\| |\xi|^s \hat{f} \big\|_1,
\end{equation}  
$ \mathcal{B}^{p,s} $ is a Banach space. We refer to \cite{MM} for the main properties of these spectral Barron spaces and their relationships with other Sobolev spaces.  

There exists a relationship between $ \mathcal{B}^{p,s} $ and $ B^s $ when $ s \ge 0 $. Specifically, we have  
\[
\mathcal{B}^{p,s} = L^p \cap B^s.  
\]  

Note that the norm \eqref{inn} is precisely the natural norm on $ L^p \cap B^s $. Moreover, we observe that $ L^p \cap B^s $ is isometrically isomorphic to the closed subspace  
\[
D := \left\{ (f, f) ; \; f \in L^p \cap B^s \right\}  
\]  
of $ L^p \times B^s $, equipped with the norm \eqref{inn}.  

In light of \cite[Subsection 4.8]{Ru}, the dual space $ (\mathcal{B}^{p,s})' $ can be identified with  
\[
\left( L^{p'} \times \tilde{B}^{-s} \right) / D^\perp,  
\]  
where  
\[
D^\perp = \left\{ (g, h) \in L^{p'} \times \tilde{B}^{-s} ; \; \langle (g, h), (f, f) \rangle = 0, \; f \in L^p \cap B^s \right\},  
\]  
and the pairing is defined as  
\[
\langle (g, h), (f, f) \rangle := \int_{\mathbb{R}^n} g(x) f(x) \, dx + \int_{\mathbb{R}^n} \left[ \langle \xi \rangle^{-s} \mathscr{F}^{-1} h(\xi) \right] \left[ \langle \xi \rangle^s \hat{f}(\xi) \right] \, d\xi.  
\]

\subsubsection{Spectral Barron Spaces associated with temperate weights.}  
Following \cite[Definition 10.1.1]{Ho}, a positive function $ K $ defined on $ \mathbb{R}^n $ is called a temperate weight function if  
\[
K(\xi + \eta) \le (1 + \mathbf{c} |\xi|)^m K(\eta), \quad \xi, \eta \in \mathbb{R}^n,  
\]  
where $ \mathbf{c} > 0 $ and $ m > 0 $ are constants. A typical example of a temperate weight function is $ K(\xi) = \langle \xi \rangle^s $ with $s \in \mathbb{R} $.  

Let $ K $ be a temperate weight function and $ 1 \le p \le \infty $. Define the space  
\[
B^{p, K} := \left\{ f \in \mathscr{S}' ; \; K \hat{f} \in L^p \right\}.  
\]  
This space is equipped with its natural norm  
\[
\|f\|_{p, K} := (2\pi)^{-n} \|K \hat{f}\|_p, \quad f \in B^{p, K}.  
\]  
This space has been introduced to study partial differential equations with constant coefficients. We refer to \cite[Chapter X]{Ho} for further details.  

Now, let $ K(\xi) = \langle \xi \rangle^s $ with $ s \in \mathbb{R} $, and replace the definition of $ B^{p, K} $ by the following:  
\[
B^{p, s} := \left\{ f \in \mathscr{S}' ; \; \langle \xi \rangle^s \hat{f} \in L^p \right\}, \quad s \in \mathbb{R}, \; 1 \le p \le \infty.  
\]  
Additionally, define  
\[
\tilde{B}^{p, s} := \left\{ f \in \mathscr{S}' ; \; \langle \xi \rangle^s \mathscr{F}^{-1} f \in L^p \right\}, \quad s \in \mathbb{R}, \; 1 \le p \le \infty.  
\]  
We proceed similarly to the case of $ B^s $ to show that the dual space of $ B^{p, s} $, where $s\ge 0$, can be identified with $ \tilde{B}^{p', -s} $, where $ p' $ denotes the Hölder conjugate exponent of $ p $ (i.e., $ \frac{1}{p} + \frac{1}{p'} = 1 $). The duality pairing between these two spaces, denoted hereafter by $ (\cdot, \cdot) $, is given by the formula  
\[
(f, g) = \int_{\mathbb{R}^n} \left[ \langle \xi \rangle^{-s} \mathscr{F}^{-1} f(\xi) \right] \left[ \langle \xi \rangle^s \hat{g}(\xi) \right] \, d\xi, \quad f \in \tilde{B}^{p', -s}, \; g \in B^{p, s}.  
\]  

\subsection{The Radon transform on $B^s$}

In this subsection, we discuss the Radon transform on $B^s$, which allows the spectral Barron space to also be characterized via a Radon measure. Furthermore, we investigate the isomorphism between the spectral Barron space and another function space.

For $\omega \in \mathbb{S}^{n-1}$ and $s \in \mathbb{R}$, the hyperplane $\{x\in \mathbb{R}^n;\; x\cdot \omega=s\}$ will be denoted henceforth by $[x \cdot \omega = s]$. The Radon transform of a function $f \in \mathscr{S}$ is defined as
\[
Rf(s, \omega) = \int_{[x \cdot \omega = s]} f(x) \, d\mu(x), \quad (s, \omega) \in \mathbb{R} \times \mathbb{S}^{n-1},
\]
where $d\mu$ denotes the Euclidean surface measure on the hyperplane $[x \cdot \omega = s]$.

For later use, we note that $Rf$ satisfies the following symmetry property:
\[
Rf(-s, -\omega) = Rf(s, \omega), \quad (s, \omega) \in \mathbb{R} \times \mathbb{S}^{n-1}.
\]
Recall that the Schwartz space $\mathscr{S}(\mathbb{R} \times \mathbb{S}^{n-1})$ consists of all $C^\infty$ functions $g$ on $\mathbb{R} \times \mathbb{S}^{n-1}$ such that, for all non-negative integers $k, \ell$ and for every differential operator $D$ on $\mathbb{S}^{n-1}$, the following estimate holds:
\[
\sup_{(t, \omega) \in \mathbb{R} \times \mathbb{S}^{n-1}} \left| (1 + |t|^k) \partial_t^\ell (Dg)(t, \omega) \right| < \infty.
\]

For $g \in \mathscr{S}(\mathbb{R} \times \mathbb{S}^{n-1})$, we denote by $\tilde{g}$ the Fourier transform of $g$ with respect to its first variable. That is,
\[
\tilde{g}(t, \omega) := \widehat{g(\cdot, \omega)}(t), \quad (t, \omega) \in \mathbb{R} \times \mathbb{S}^{n-1}. 
\]
We define the weighted $L^1$ space $L^1_s := L^1\left(\mathbb{R} \times \mathbb{S}^{n-1}, \frac{1}{2}|t|^{n-1} \langle t\omega \rangle^s \, dt \, d\omega\right)$ and equip it with its natural norm
\[
\|\cdot\|_{L^1_s} := \|\cdot\|_{L^1\left(\mathbb{R} \times \mathbb{S}^{n-1}, \frac{1}{2}|t|^{n-1} \langle t\omega \rangle^s \, dt \, d\omega\right)}.
\]

We now define the space
\[
\mathcal{L}^s := \left\{ g \in C_b^0(\mathbb{R}, L^1(\mathbb{S}^{n-1})) ; \; \tilde{g} \in L^1_s \right\},
\]
equipped with the norm
\[
\|g\|_{\mathcal{L}^s} := \|\tilde{g}\|_{L^1_s}.
\]
As we have previously verified for a similar space $B^s$, we confirm that $\mathcal{L}^s$ is a Banach space with respect to the norm $\|\cdot\|_{\mathcal{L}^s}$.

Let us denote the subspace of functions $g \in \mathscr{S}(\mathbb{R} \times \mathbb{S}^{n - 1})$ that satisfy the symmetry condition $g(-t, -\omega) = g(t, \omega)$ by $\tilde{\mathscr{S}}(\mathbb{R} \times \mathbb{S}^{n - 1})$. We then define $\mathscr{S}_{\mathrm{ho}}(\mathbb{R} \times \mathbb{S}^{n - 1})$ as the subspace of functions $g \in \tilde{\mathscr{S}}(\mathbb{R} \times \mathbb{S}^{n - 1})$ with the following property: for all integers $k \geq 0$, the integral $\int_{\mathbb{R}} g(t, \omega) t^k dt$ is a homogeneous polynomial in the components $\omega_1,\ldots,\omega_n$ of $\omega$ of degree $k$.
According to \cite[Theorem 2.4]{He}, the operator $R$ is a one-to-one linear mapping from the space $\mathscr{S}$ onto the space $\mathscr{S}_{\mathrm{ho}}(\mathbb{R} \times \mathbb{S}^{n - 1})$. 

Let $\mathcal{K}^s$ be defined as the closure of the subspace $\mathscr{S}_{\mathrm{ho}}(\mathbb{R} \times \mathbb{S}^{n - 1})$ within the space $\mathcal{L}^s$. We have the following result. 

\begin{proposition}\label{prora1}
The Radon transform $R$ can be extended to an isometric isomorphism from $B^s$ onto $\mathcal{K}^s$.
\end{proposition}

\begin{proof}
Let $f\in\mathscr{S}$. According to \cite{He}, we have the formula
\begin{equation}\label{ra1}
\hat{f}(t\omega)=\widetilde{R f}(t,\omega),\; \quad (t,\omega)\in \mathbb{R}\times \mathbb{S}^{n-1}.
\end{equation}

For $s\geq0$, we know that
\begin{equation}\label{ra2}
\int_{\mathbb{R}^n}\langle \xi\rangle ^s |\hat{f}(\xi)|d\xi=\int_0^{+\infty}\int_{\mathbb{S}^{n-1}}t^{n-1} \langle t\omega \rangle^s |\hat{f}(t\omega)|dt d\omega .
\end{equation}

By making the change of variables $(t,\omega)\to(-t,-\omega)$ in the right-hand integral of \eqref{ra2}, we get
\begin{equation}\label{ra3}
\int_{\mathbb{R}^n}\langle \xi\rangle ^s |\hat{f}(\xi)|d\xi=\int_{-\infty}^0\int_{\mathbb{S}^{n-1}}(-t)^{n-1} \langle t\omega \rangle^s |\hat{f}(t\omega)|dt d\omega .
\end{equation}

Combining \eqref{ra2} and \eqref{ra3}, we obtain
\[
\int_{\mathbb{R}^n}\langle \xi\rangle ^s |\hat{f}(\xi)|d\xi=\frac{1}{2}\int_{\mathbb{R}}\int_{\mathbb{S}^{n-1}}|t|^{n-1} \langle t\omega \rangle^s |\hat{f}(t\omega)|dt d\omega .
\]
Using \eqref{ra1}, we can rewrite the above identity as
\[
\int_{\mathbb{R}^n}\langle \xi\rangle ^s |\hat{f}(\xi)|d\xi=\frac{1}{2}\int_{\mathbb{R}}\int_{\mathbb{S}^{n-1}}|t|^{n-1} \langle t\omega \rangle^s |\widetilde{Rf}(t,\omega)|dt d\omega .
\]
In other words, we have shown that
\begin{equation}\label{ra4}
\|f\|_{B^s}=\|Rf \|_{\mathcal{L}^s}, \quad f\in  \mathscr{S}.
\end{equation}

Now, take $g\in B^s$ and let $(g_j)$ be a sequence in $\mathscr{S}$ that converges to $g$ in $B^s$. By \eqref{ra4}, $(Rg_j)$ is a Cauchy sequence in $\mathcal{L}^s$. Since $\mathcal{L}^s$ is complete, there exists a limit $Rg\in\mathcal{L}^s$ of the sequence $(Rg_j)$. We can verify that this limit is independent of the sequence $(g_j)$ converging to $g$ in $B^s$. Taking $f=g_j$ in \eqref{ra4} and passing to limit as $j$ goes to $\infty$, we conclude that $R:B^s \rightarrow \mathcal{L}^s$ is an isometry. Since $R$ is a one-to-one linear mapping from $\mathscr{S}$ onto $\mathscr{S}_{\mathrm{ho}}(\mathbb{R}\times\mathbb{S}^{n - 1})$ and $\mathcal{K}^s$ is the closure of $\mathscr{S}_{\mathrm{ho}}(\mathbb{R}\times\mathbb{S}^{n - 1})$ in $\mathcal{L}^s$, we conclude that $R:B^s\to\mathcal{K}^s$ is onto. The proof is complete.
\end{proof}

For extended discussion on the connections between the Radon transform and spectral Barron spaces, we refer to \cite{KH}, which provides further results on isomorphisms of the Radon transform in relation to general Bessel-potential spaces and total variation norms. Additionally, \cite{LuM} addresses the inverse problem associated with the Radon transform, where a linking condition is established between the spectral Barron space and the sinogram spaces, thereby characterizing the range of the Radon transform.

\section{Schr\"odinger equations in spectral Barron spaces on $\mathbb{R}^n$} 
In this section, we present a comprehensive discussion of Schr\"odinger equations in spectral Barron spaces on $\mathbb{R}^n$, encompassing regularity results for both the standard Laplacian operator and anisotropic elliptic operators.

In the following, $\|\cdot\|_{\mathscr{B}(E)}$ will denote the operator norm on $\mathscr{B}(E)$, the space of bounded linear operators from the Banach space $E$ to itself.

\subsection{The Laplacian operator in $B^0$}

Define the unbounded operator $ A: B^0 \rightarrow B^0 $ by
\[
Au := -\Delta u, \quad u \in D(A) := \{ u \in B^0; \; \Delta u \in B^0 \}.
\]  
It is standard to equip $D(A)$ with its natural norm  
\[
\| u \|_{D(A)} := \| u \|_{B^0} + \| Au \|_{B^0}, \quad u \in D(A).
\]  
Equivalently, in terms of the Fourier transform, we have
\[
\| u \|_{D(A)} = \| \langle \xi \rangle^2 \hat{u} \|_1, \quad u \in D(A).  
\]  
This implies that $ D(A) = B^2 $ and $ \| \cdot \|_{D(A)} = \| \cdot \|_{B^2} $.  

Recall that the resolvent set of $ A $ is defined as
\[
\rho(A) := \{ z \in \mathbb{C} ; \; z - A : B^0 \rightarrow B^0 \text{ is an isomorphism} \}.  
\]  
The following properties are important for our subsequent analysis.

\begin{proposition}\label{pro1}
The operator $ A $ satisfies the following properties:
\\
$\mathrm{(i)}$ The domain $ D(A) $ is dense in $ B^0 $. 
\\
$\mathrm{(ii)}$ The operator $ A $ is closed.  
\\
$\mathrm{(iii)}$ The interval $ (-\infty, 0) $ is contained in the resolvent set $ \rho(A) $, and  
\begin{equation}\label{0}
\sup_{t > 0} \| t(t + A)^{-1} \|_{\mathscr{B}(B^0)} \leq 1.  
\end{equation}
\end{proposition}

\begin{proof}

(i) Since $ \mathscr{S} \subset D(A) $ and $ \mathscr{S} $ is dense in $ B^0 $ by Lemma \ref{lem1.1}, it follows that $ D(A) $ is dense in $ B^0 $.  

(ii) Let $ (u_j) $ be a sequence in $ D(A) $ such that $ u_j \to u $ in $ B^0 $ and $ Au_j \to v $ in $ B^0 $. Then, in the Fourier domain: $ \hat{u}_j \to \hat{u} $ in $ L^1 $ and $ |\xi|^2 \hat{u}_j \to \hat{v} $ in $ L^1 $.  

Since $ |\xi|^2 \hat{u}_j \to |\xi|^2 \hat{u} $ in $ \mathscr{S}' $, it follows that $ |\xi|^2 \hat{u} = \hat{v} $. Thus, $ u \in D(A) $ and $ Au = v $, proving that $ A $ is closed.  

(iii) Let $ t < 0 $, $ f \in B^0 $, and consider the equation
\[
(t - A)u = f.  
\]  
If this equation admits a solution $ u \in D(A) $, then in the Fourier domain, the solution must satisfy
\[
(t - |\xi|^2)\hat{u} = \hat{f}.  
\]  
Thus, the unique solution $ u \in D(A) $ is given by
\[
u = \mathscr{F}^{-1}\left( (t - |\xi|^2)^{-1} \hat{f} \right).  
\]  
This implies that $ (-\infty, 0) \subset \rho(A) $. For $ t < 0 $, the resolvent operator satisfies
\[
\| (t - A)^{-1} f \|_{B^0} = \left\| (t - |\xi|^2)^{-1} \hat{f} \right\|_1.  
\]  
On the other hand, for $ t > 0 $, we have
\[
\| t(t + A)^{-1} f \|_{B^0} = \left\| t(t + |\xi|^2)^{-1} \hat{f} \right\|_1 \leq \| \hat{f} \|_1 = \| f \|_{B^0}.  
\]  

This completes the proof. 
\end{proof}

In light of Proposition \ref{pro1}, we apply the Hille-Yosida theorem (e.g., \cite[Theorem 3.1]{Pa}) to deduce that $ -A $ generates a strongly continuous semigroup of contractions $ (\mathbb{T}(t))_{t \geq 0} $. Furthermore, we have the following representation formula for the resolvent:  
\[
(\lambda + A)^{-1} = \int_0^\infty e^{-\lambda t} \mathbb{T}(t) \, dt, \quad \lambda > 0.  
\]  

We now recall the definition of a sectorial operator. For $ \omega \in (0, \pi) $, we define the sector $ S_\omega $ as
\[
S_\omega := \{ z \in \mathbb{C} \setminus \{0\} ;\; |\arg z| < \omega \}.  
\]  
An operator $ \mathbf{A} $ on a Banach space $ E $ is called sectorial of angle $ \omega $ if
there holds $ \sigma(\mathbf{A}):=\mathbb{C}\setminus \rho(A) \subset \overline{S}_\omega $ (the closure of $ S_\omega $), and 
\[
\sup_{\lambda \in \mathbb{C} \setminus \overline{S}_\omega} \| \lambda (\lambda - \mathbf{A})^{-1} \|_{\mathscr{B}(E)} < \infty.  
\]  

A combination of Proposition \ref{pro1} and \cite[Proposition 2.1.1]{Ha} yields the following result.

\begin{corollary}\label{cor1}  
The operator $ A $ is sectorial of angle $ \pi/2 $.  
\end{corollary}  

As it was already observed in \cite{LuM}, since $ A $ is sectorial, it generates an analytic semigroup (e.g., \cite[Chapter 2]{Lun}), and the standard functional calculus applies to $ A $ (e.g., \cite{Ha}). Here, we briefly discuss the fractional powers of the operator $ (1 - \Delta, D(A)) $.  
Consider the operator $ \mathbb{A} := 1 - \Delta $ with domain $ D(\mathbb{A}) = D(A) $. By Proposition \ref{pro1}, $ \mathbb{A} $ is a closed and densely defined operator such that $ (-\infty, 0] \subset \rho(\mathbb{A}) $, and the following estimate holds:  
\[
\sup_{t \geq 0} (1 + t) \| (t + \mathbb{A})^{-1} \|_{\mathscr{B}(B^0)} \leq 1.  
\]  
Thus, $ \mathbb{A} $ is a positive operator in the sense of \cite[Definition 4.1]{Lu}.  

Define the regions
\[
\Lambda := \{ z \in \mathbb{C}; \; \Re z \leq 0, \; |\Im z| < |\Re z| + 1 \} \cup \{ z \in \mathbb{C}; \; |z| < 1 \}
\]  
and, for $ (\delta, \theta) \in (0, 1) \times (0, \pi/4) $,  
\[
\Lambda_{\delta, \theta} := \{ z \in \mathbb{C}; \; \Re z < 0, \; |\Im z| \leq (\tan \theta) |\Re z|, \; |z| \le \delta \}.
\]  
According to \cite[Lemma 4.2]{Lu}, $ \Lambda \subset \rho(\mathbb{A}) $, and for all $ (\delta, \theta) \in (0, 1) \times (0, \pi/4) $, there exists a constant $ M_{\delta, \theta} $ such that
\[
\sup_{z \in \Lambda_{\delta, \theta}} (1 + |z|) \| (z - \mathbb{A})^{-1} \|_{\mathscr{B}(B^0)} \leq M_{\delta, \theta}.  
\]  

Let $ 0 < \alpha < 1 $. By the formula \cite[(4.4)]{Lu}, we have
\[
\mathbb{A}^{-\alpha} f = \frac{\sin (\pi \alpha)}{\pi} \int_0^\infty \lambda^{-\alpha} (\lambda + \mathbb{A})^{-1} f \, d\lambda, \quad f \in B^0.  
\]  
Let $ f \in B^0 $ and define $ g(\lambda, \cdot) := (\lambda + \mathbb{A})^{-1} f $. Then, the Fourier transform of $ g(\lambda, \cdot) $ is given by
\[
\widehat{g(\lambda, \cdot)}(\xi) = (\lambda + \langle \xi \rangle^2)^{-1} \hat{f}(\xi).
\]   
Using the integral formula
\[
\int_0^\infty \lambda^{-\alpha} (\lambda + \langle \xi \rangle^2)^{-1} \, d\lambda = \frac{\pi}{\sin (\pi \alpha)} \langle \xi \rangle^{-2\alpha},  
\]  
we deduce that 
\[
\frac{\sin (\pi \alpha)}{\pi}\langle \xi \rangle^{2\alpha}\int_0^\infty \lambda^{-\alpha} \widehat{g(\lambda, \cdot)}(\xi) \, d\lambda=\hat{f}(\xi)
\]
and
$ \langle \xi \rangle^{2\alpha} \lambda^{-\alpha}\widehat{g(\lambda, \cdot)}(\xi) $ belongs to $ L^1((0, \infty) \times \mathbb{R}^n) $.  
In other words, we have shown that
\begin{equation}\label{fr1}
\langle\xi\rangle^{2\alpha}\mathscr{F}(\mathbb{A}^{-\alpha}f)=\hat{f}
\end{equation}
and $\| \mathbb{A}^{-\alpha} f \|_{B^{2\alpha}} = \| f \|_{B^0}$.

The unbounded operator $ \mathbb{A}^\alpha $ is defined as follows:  
\[
\mathbb{A}^\alpha f := \mathbb{A} \mathbb{A}^{-(1 - \alpha)} f, \quad f \in D(\mathbb{A}^\alpha) := \{ h \in B^0; \; \mathbb{A} \mathbb{A}^{-(1 - \alpha)} h \in B^0 \}.
\]  
Let $ f \in D(\mathbb{A}^\alpha) $. Since
\[
\langle \xi \rangle^2 \mathscr{F}(\mathbb{A}^{-(1 - \alpha)} f) = \langle \xi \rangle^{2\alpha} \left[ \langle \xi \rangle^{2(1 - \alpha)} \mathscr{F}(\mathbb{A}^{-(1 - \alpha)} f) \right]=\langle\xi\rangle^{2\alpha}\hat{f}
\]
by \eqref{fr1} and $ \langle \xi \rangle^2 \mathscr{F}(\mathbb{A}^{-(1 - \alpha)} f) \in L^1 $, it follows that $ f \in B^{2\alpha} $. Conversely, if $ f \in B^{2\alpha} $, we can adapt the preceding proof to show that $ \mathbb{A}^{-(1 - \alpha)} f \in B^2 $, and thus $ D(\mathbb{A}^{\alpha}) = B^{2\alpha} $.  
From its definition, $ \mathbb{A}^\alpha $ can be expressed as
\[
\mathbb{A}^\alpha f = \frac{\sin (\pi \alpha)}{\pi} \mathbb{A} \int_0^\infty \lambda^{-1 + \alpha} (\lambda + \mathbb{A})^{-1} f \, d\lambda, \quad f \in B^{2\alpha}.  
\]

Even though we are not in a Hilbertian setting, the results in this subsection show that $-\Delta$, as an unbounded operator on $B^0$, shares properties similar to those in the $L^2$ setting. We also note that \cite{LuM} has proven that $1-\Delta$ is weakly sectorial in $B^0$, which suffices to define its fractional powers.

\subsection{Schr\"odinger equations}

We study the Schr\"odinger equations within the framework of spectral Barron spaces defined over the entire domain $\mathbb{R}^n$. Specifically, we first analyze a special case involving a potential with small amplitude and conduct a detailed spectral analysis. Additionally, we will explore Schr\"odinger equations incorporating non-local terms.

\subsubsection{The case of small norm potential}

The Schr\"odinger equation considered in this subsection is given by
\begin{equation}\label{Sc}
(1 - \Delta + V)u = f,
\end{equation}
where both $ V, f \in B^0 $. It is important to note that we cannot replace $ 1 + V $ with a function belonging to $ B^0 $, as $ 1 \notin B^0 $, which we have previously observed in Remark \ref{remB0}. However, the constant $ 1 $ can be replaced by an arbitrary strictly positive constant.
We introduce the following set
\[
\mathscr{V}_0 := \left\{ V \in B^0; \; (2\pi)^{-n} \|V\|_{B^0} < 1 \right\}.
\]

The well-posedness and regularity results for the Schr\"odinger equation (\ref{Sc}) are presented below.
\begin{theorem}\label{thm1}
For all $ V \in \mathscr{V}_0 $ and $ f \in B^0 $, the Schrödinger equation \eqref{Sc} admits a unique solution $ \mathbf{u} \in B^2 $, and the following estimate holds:
\begin{equation}\label{00}
\|\mathbf{u}\|_{B^2} \leq (1 - \varkappa)^{-1} \|f\|_{B^0},
\end{equation}
where $ \varkappa := (2\pi)^{-n} \|V\|_{B^0} $.
\end{theorem}

\begin{proof}
We rewrite the Schr\"odinger equation \eqref{Sc} in the form
\[
(1 - \Delta)u = -Vu + f.
\]
Given that $-1 \in \rho(A)$, the last equation yields
\[
u = (1 + A)^{-1}(-Vu + f) =: Fu,
\]
where $ F: B^0 \to B^0 $. From \eqref{0} and \eqref{prod}, it follows that for all $ u, v \in B^0 $
\[
\|Fu - Fv\|_{B^0} \leq (2\pi)^{-n} \|V\|_{B^0} \|u - v\|_{B^0} = \varkappa \|u - v\|_{B^0},
\]
where $ \varkappa := (2\pi)^{-n} \|V\|_{B^0} $. Since $ \varkappa < 1 $, $ F $ is a contraction mapping. By the Banach contraction principle, there exists a unique $ \mathbf{u} \in B^0 $ satisfying $ \mathbf{u} = F\mathbf{u} $. Consequently, $ \mathbf{u} \in D(A) = B^2 $ is the unique solution of \eqref{Sc}.

Furthermore, since
\[
\|\mathbf{u}\|_{B^0} \leq \|f\|_{B^0} + \|F\mathbf{u} - F0\|_{B^0},
\]
we obtain
\begin{equation}\label{est0}
\|\mathbf{u}\|_{B^0} \leq (1 - \varkappa)^{-1} \|f\|_{B^0}.
\end{equation}
On the other hand, using
\[
\langle \xi \rangle^2 \widehat{\mathbf{u}} = -\widehat{V\mathbf{u}} + \hat{f},
\]
and Proposition \ref{pro2} (iv), we derive
\[
\|\mathbf{u}\|_{B^2} \leq (2\pi)^{-n} \|V\|_{B^0} \|\mathbf{u}\|_{B^0} + \|f\|_{B^0}.
\]
Combining this with \eqref{est0} yields \eqref{00}.
\end{proof}

\begin{remark}\label{rem2}
{\rm
The above well-posedness and regularity theorem was established in \cite[Theorem 2.3]{CLLZ} under the assumptions \( V \in B^0 \) and \( V \ge -1 \), using a compactness argument based on the Fourier transform. In our approach, we apply the Banach contraction principle, which allows for an extension to more general cases to be discussed in the following (sub)section, including the anisotropic elliptic equation.
}
\end{remark}

Next, we prove the following higher regularity result.
\begin{theorem}\label{thm2}
Let $ k \ge 0 $ be an integer, and let $ V \in \mathscr{V}_0 \cap B^k $. There exists a constant $ \mathbf{c} = \mathbf{c}(n, k, V) > 0 $ such that for all $ f \in B^k $, the solution $ \mathbf{u} \in B^2 $ of \eqref{Sc} satisfies $ \mathbf{u} \in B^{k+2} $, and the following estimate holds:
\[
\|\mathbf{u}\|_{B^{k+2}} \leq \mathbf{c} \|f\|_{B^k}.
\]
\end{theorem}

\begin{proof}
We prove Theorem \ref{thm2} by induction. The case $k=0$ follows from Theorem \ref{thm1} with $\mathbf{c}:=(1-\varkappa)^{-1}$. We assume that Theorem \ref{thm2} holds when $k\ge 0$. 

Using the Fourier representation
\[
\langle \xi \rangle^2 \widehat{\mathbf{u}} = -\widehat{V\mathbf{u}} + \hat{f},
\]
we obtain
\[
\langle \xi \rangle^{k+3} \widehat{\mathbf{u}} = -\langle\xi\rangle^{k+1}\widehat{V\mathbf{u}} + \langle\xi\rangle^{k+1}\hat{f}.
\]
Hence, we obtain from \eqref{prod}
\[
\|\mathbf{u}\|_{B^{k+3}} \le 2^{(k+1)/2} (2\pi)^{-n} \|V\|_{B^{k+1}} \|\mathbf{u}\|_{B^{k+1}} + \|f\|_{B^{k+1}}.
\]
From the assumption by induction, we have
\begin{align*}
\|\mathbf{u}\|_{B^{k+3}} &\le 2^{(k+1)/2} (2\pi)^{-n} \mathbf{c}\|V\|_{B^{k+1}} \|f\|_{B^k} + \|f\|_{B^{k+1}}
\\
&\le \left(2^{(k+1)/2} (2\pi)^{-n} \mathbf{c}\|V\|_{B^{k+1}} +1\right)\|f\|_{B^{k+1}},
\end{align*}
where $\mathbf{c}=\mathbf{c}(n,k,V)>0$ is a constant. By setting the generic constant as
\[
\mathbf{c}(n,k+1,V):=2^{(k+1)/2} (2\pi)^{-n} \mathbf{c}(n,k,V)\|V\|_{B^{k+1}} +1,
\]
we complete the proof.
\end{proof}

To establish the $ B^{s+2} $-regularity for solutions of \eqref{Sc} when $ s \in (0, \infty) \setminus \mathbb{N} $, it is necessary to solve the equation \eqref{Sc} directly. For this purpose, we introduce the following set
\[
\mathscr{V}_s := \left\{ V \in B^s ; \; 2^{s/2} (2\pi)^{-n} \|V\|_{B^s} < 1 \right\}.
\]
By adapting the methodology employed in the proof of Theorem \ref{thm1}, we derive the following result.

\begin{theorem}\label{thm2.b}
Let $ s \geq 0 $ be a real number and $ V \in \mathscr{V}_s $. For every $ f \in B^s $, the equation \eqref{Sc} admits a unique solution $ \mathbf{u} \in B^{s+2} $ satisfying the estimate
\[
\|\mathbf{u}\|_{B^{s+2}} \leq \mathbf{c} \|f\|_{B^s},
\]
where $\mathbf{c}:=(1-2^{s/2} (2\pi)^{-n} \|V\|_{B^s})^{-1}$ is a constant.
\end{theorem}

\subsubsection{Spectral analysis for Schr\"odinger equations}

To investigate the Schr\"odinger equation without assuming a smallness condition on the potential, we conduct a spectral analysis of Schr\"odinger equations within the framework of spectral Barron spaces. 

Let $ \lambda \in \mathbb{C} \setminus \{0\} $ and $ V \in B^0 \setminus \{0\} $. Consider the eigenvalue problem
\begin{equation}\label{sp0}
(1-\Delta +\lambda V)u=0.
\end{equation}
We reformulate \eqref{sp0} into the following equation:
\begin{equation}\label{sp2}
Tu:=(1-\Delta )^{-1}Vu= -\lambda^{-1}u.
\end{equation}
Let $ \|T\| $ denote the operator norm of $ T $ in $ \mathscr{B}(B^0) $. Then, by the definition of the $ B^0 $ norm, we have
\[
\|T\| \leq (2\pi)^{-n} \|V\|_{B^0} =: \varkappa.
\]
Since $ T $ is compact by Proposition \ref{procm2}, it follows from \cite[Proposition 1.8]{Ch} that the spectrum $ \sigma(T) $ is contained in the ball $B(0, \varkappa)$. 
Moreover, the spectrum $ \sigma(T) $ is either trivial ($ \sigma(T) = \{0\} $) or its non-zero part $ \sigma(T) \setminus \{0\} $ consists of isolated eigenvalues (e.g., \cite[Lemma 1.10]{Ch}).

\begin{theorem}\label{thmsp}
Let $ V \in B^0 \setminus \{0\} $. For all $ -\lambda^{-1} \in \rho(T) = \mathbb{C} \setminus \sigma(T) $ and $ f \in B^0 $, the equation 
\[
(1 - \Delta + \lambda V) u = f
\]
admits a unique solution $ u \in B^2 $. Moreover, if $ V \in B^j $ and $ f \in B^j $ for some integer $ j \geq 1 $, then $ u \in B^{j+2} $, and the following estimate holds:
\begin{equation}\label{tsp}
\|u\|_{B^{j+2}}\le \gamma_j\left(|\lambda|\|V\|_{B^j}+1\right)^{j+1}\|f\|_{B^j},
\end{equation} 
where $ \gamma_j=\gamma_j(n)>0 $ is a  constant.
\end{theorem}

\begin{proof}
Let $ -\lambda^{-1} \in \rho(T) $. Then, for all $ f \in B^0 $, there exists a unique solution $ u \in B^0 $ satisfying
\begin{equation}\label{sp3}
Tu=-\lambda^{-1}u+f.
\end{equation}
Let $ (f_j) $ be a sequence in $ \mathscr{S} $ converging to $ f $ in $ B^0 $, and let $ u_j \in B^0 $ be the solution of \eqref{sp3} when $ f $ is replaced by $ f_j $. Since $ T u_j \in B^2 $, it follows that $ u_j \in B^2 $. 
Next, observe that
\begin{equation}\label{sp4}
\|u_j-u_k\|_{B^0}\le \|(T+\lambda^{-1})^{-1}\|\|f_j-f_k\|_{B^0}.
\end{equation}
Thus, $ (u_j) $ is a Cauchy sequence in $ B^0 $. Let $ v \in B^0 $ be the limit of $ (u_j) $ in $ B^0 $. Taking $ f = f_j $ in \eqref{sp3} and passing to the limit as $ j \to \infty $, we obtain $ v = u $. 

Now, using the identity
\[
(1 - \Delta)(u_j - u_k) = -\lambda V (u_j - u_k) + (f_j - f_k),
\]
we derive in the Fourier domain
\[
\langle \xi \rangle^2 \mathscr{F}(u_j - u_k) = -\lambda \mathscr{F}(V (u_j - u_k)) + \mathscr{F}(f_j - f_k).
\]
Taking norms, we get
\[
\|u_j - u_k\|_{B^2} \leq (2\pi)^{-n} |\lambda| \|V\|_{B^0} \|u_j - u_k\|_{B^0} + \|f_j - f_k\|_{B^0}.
\]
This shows that $ (u_j) $ is a Cauchy sequence in $ B^2 $, and hence $ u \in B^2 $. Moreover, we have
\begin{equation}\label{sp5}
\|u\|_{B^2}\le \mathbf{c}_0\|f\|_{B^0},
\end{equation}
where
\[
\mathbf{c}_0 := (2\pi)^{-n} |\lambda| \|V\|_{B^0} \|(T + \lambda^{-1})^{-1}\| + 1.
\]
Now, assume $ V \in B^1 \setminus \{0\} $ and $ f \in B^1 $. Then $ -\lambda V u + f \in B^1 $ and
\[
\|-\lambda V u + f\|_{B^1} \leq \sqrt{2} (2\pi)^{-n} |\lambda| \|V\|_{B^1} \|u\|_{B^1} + \|f\|_{B^1}.
\]

Using \eqref{sp5}, we obtain
\begin{equation}\label{sp6}
\|-\lambda Vu+f\|_{B^1}\le \left(\sqrt{2}(2\pi)^{-n}|\lambda|\|V\|_{B^1}\mathbf{c}_0+1\right)\|f\|_{B^1}.
\end{equation}
From the Fourier relation
\[
\langle \xi \rangle^3 \hat{u} = \langle \xi \rangle \mathscr{F}(-\lambda V u + f),
\]
and \eqref{sp6}, we deduce
\begin{equation}\label{sp7}
\|u\|_{B^3}\le \mathbf{c}_1\|f\|_{B^1},
\end{equation}
where
\[
\mathbf{c}_1 := \sqrt{2} (2\pi)^{-n} |\lambda| \|V\|_{B^1} \mathbf{c}_0 + 1.
\]

For general $ j \geq 1 $, assume $ V \in B^j \setminus \{0\} $ and $ f \in B^j $. Define
\[
\mathbf{c}_j := 2^{j/2} (2\pi)^{-n} |\lambda| \|V\|_{B^j} \mathbf{c}_{j-1} + 1.
\]
Using the inequality
\[
\|-\lambda V u + f\|_{B^j} \leq 2^{j/2} (2\pi)^{-n} |\lambda| \|V\|_{B^j} \|u\|_{B^j} + \|f\|_{B^j},
\]
an induction argument shows that
\begin{equation}\label{sp8}
\|u\|_{B^{j+2}}\le \mathbf{c}_j\|f\|_{B^j}.
\end{equation}
Finally, \eqref{tsp} follows from \eqref{sp8} by noting that $ \mathbf{c}_j $ can be bounded by \[ \gamma_j \left( |\lambda| \|V\|_{B^j} + 1 \right)^{j+1} \] for some constant $ \gamma_j=\gamma_j(n)>0 $. This completes the proof.
\end{proof}

\begin{theorem}\label{thmsp2}
Let $ V \in B^0 \setminus \{0\} $ and $ -\lambda^{-1} \in \sigma(T) \setminus \{0\} $. Then, the following hold:
\\
$\mathrm{(i)}$ The null space $ N(T + \lambda^{-1}) $ is finite-dimensional.
\\
$\mathrm{(ii)}$ Each $ u \in N(T + \lambda^{-1}) $ belongs to $ B^2 $ and is a solution to the equation
\[
(T + \lambda^{-1})u = 0.
\]
$\mathrm{(iii)}$ If $ V \in B^j $ for some $ j \geq 1 $ and $ u \in N(T + \lambda^{-1}) $, then $ u \in B^{j+2} $ and the following inequality holds:
\begin{equation}\label{sp9}
\|u\|_{B^{j+2}} \leq 2^{j(j+1)/4} \left[(2\pi)^{-n} |\lambda| \|V\|_{B^j}\right]^j \|u\|_{B^0}.
\end{equation}
\end{theorem}

\begin{proof}
The finite-dimensionality of $ N(T + \lambda^{-1}) $ follows from the compactness of the operator $ T $ (e.g., \cite[Theorem 1.5]{Ch}). The rest of the proof is similar to that of Theorem \ref{thmsp}.
\end{proof}

\subsubsection{Schr\"odinger equations with a non local term}

We adopt the same notation as in the preceding subsection. Let $ V \in B^0 \setminus \{0\} $, define $ T = (1 - \Delta)^{-1} V $, and assume $ -\lambda^{-1} \in \rho(T) $. Fix $ \mathbf{k} \in L^1 $, $ f \in B^0 $, and consider the equation
\begin{equation}\label{nl1}
(1-\Delta +\lambda V)u=\mathbf{k}\ast u+f.
\end{equation}
A function $ u \in B^2 $ is a solution to \eqref{nl1} if and only if it satisfies
\begin{equation}\label{nl2}
u=\lambda^{-1}(\lambda^{-1}+T)^{-1}(1-\Delta)^{-1}[(\mathbf{k}\ast u)+f ]=:\mathbf{K}u.
\end{equation}
We recall that $\|\cdot\|$ denotes the operator norm on $ \mathscr{B}(B^0)$. Since $ \|(1 - \Delta)^{-1}\| \leq 1 $, applying Lemma \ref{lemconv} yields
\[
\|\mathbf{K} u - \mathbf{K} v\|_{B^0} \leq |\lambda|^{-1}\|(\lambda^{-1} + T)^{-1}\| \|\mathbf{k}\|_1 \|u - v\|_{B^0}, \quad u, v \in B^0.
\]
Thus, under the condition
\begin{equation}\label{nl3}
\delta:=|\lambda|^{-1}\|(\lambda^{-1}+T)^{-1}\|\|\mathbf{k}\|_1<1,
\end{equation}
the operator $ \mathbf{K}: B^0 \to B^0 $ is contractive. By the Banach contraction principle, there exists a unique solution $ u^\ast \in B^0 $ such that $ u^\ast = \mathbf{K} u^\ast $. Consequently, we get
\begin{equation}\label{nl4}
\|u\|_{B^0}\le (1-\delta)^{-1}\|f\|_{B^0}.
\end{equation}
Proceeding analogously to the proof of Theorem \ref{thmsp}, we deduce that $ u^\ast \in B^2 $ and is the unique solution to \eqref{nl1}. In Fourier space, this implies
\[
\langle \xi \rangle^2 \widehat{u^\ast} = -\lambda (2\pi)^{-n} \widehat{V} \ast \widehat{u^\ast} + \widehat{\mathbf{k}} \widehat{u^\ast} + \widehat{f},
\]
and thus
\[
\|u^\ast\|_{B^2} \leq \left( |\lambda|(2\pi)^{-n} \|V\|_{B^0} + \|\mathbf{k}\|_1 \right) \|u^\ast\|_{B^0} + \|f\|_{B^0}.
\]
Combining this with \eqref{nl4}, we obtain
\begin{equation}\label{nl5}
\|u^\ast\|_{B^2}\le \left(\left[|\lambda|(2\pi)^{-n}\|V\|_{B^0}+\|\mathbf{k}\|_1\right](1-\delta)^{-1}+1\right)\|f\|_{B^0}.
\end{equation}
In summary, we state the following result.

\begin{theorem}\label{thmnl}
Under the assumption \eqref{nl3}, the equation \eqref{nl1} admits a unique solution $ u^\ast \in B^2 $ satisfying \eqref{nl5}.
\end{theorem}

\subsection{Anisotropic Schr\"odinger equations}\label{subsec_aniso}

In this subsection, we study anisotropic Schr\"odinger equations in spectral Barron spaces. Let $ \mu > 0 $, $ \mathfrak{A}^0 = (\mathfrak{a}_{k\ell}^0) $ be a symmetric matrix, and $ \mathfrak{A} = (\mathfrak{a}_{k\ell}) $ be a matrix-valued function defined on $ \mathbb{R}^n $.

\begin{assumption}\label{assp_aniso}
We assume that $\mathfrak{A}^0 $ and $\mathfrak{A}$ satisfy following conditions  
\begin{enumerate}
\item $ \mathfrak{a}_{k\ell} \in C_b^0 $ and $ \mathfrak{a}_{k\ell} - \mathfrak{a}_{k\ell}^0 \in B^1 $ for all $ k, \ell $.  

\item There exists $ \kappa > 0 $ such that  
\[
\mathfrak{A}^0 \xi \cdot \xi \geq \kappa |\xi|^2, \quad \xi \in \mathbb{R}^n.
\]  

\item The following inequality holds:  
\begin{equation}\label{ani0}
\eta := \sqrt{2}\varkappa (2\pi)^{-n} \sum_{k,\ell} \|\mathfrak{a}_{k\ell} - \mathfrak{a}_{k\ell}^0\|_{B^1} < 1,
\end{equation}  
where $ \varkappa := 1 / \min(\mu, \kappa) $.  
\end{enumerate}
\end{assumption}

Define the operator $ L_0: B^2 \to B^0 $ by  
\[
L_0 u = \mu u - \sum_{k,\ell} \mathfrak{a}_{k\ell}^0 \partial_{k\ell}^2 u, \quad u \in B^2.
\]  
In Fourier space, there holds
\[
\widehat{L_0 u} = (\mu + \mathfrak{A}^0 \xi \cdot \xi) \hat{u}
\]
and
\[
\min(\mu,\kappa)\langle \xi\rangle^2|\hat{u}|\le (\mu+\kappa |\xi|^2)|\hat{u}|\le |\widehat{L_0u}|.
\]
Therefore, $L_0:B^2\rightarrow B^0$ is invertible and we have
\[
\|L_0^{-1} f\|_{B^2} \leq \varkappa \|f\|_{B^0}, \quad f \in B^0.
\]  

Next, define $ L: B^2 \to B^0 $ by  
\begin{equation}\label{opeL}
L u = \mu u - \sum_{k,\ell} \partial_k (\mathfrak{a}_{k\ell} \partial_\ell u), \quad u \in B^2.
\end{equation}  
The difference $ L - L_0 $ is given by  
\[
(L - L_0) u = -\sum_{k,\ell} \partial_k \left( [\mathfrak{a}_{k\ell} - \mathfrak{a}_{k\ell}^0] \partial_\ell u \right), \quad u \in B^2.
\]  
In Fourier space, this becomes  
\[
\mathscr{F}((L - L_0) u) = -i \sum_{k,\ell} \xi_k \mathscr{F}([\mathfrak{a}_{k\ell} - \mathfrak{a}_{k\ell}^0] \partial_\ell u).
\]  
Thus,  
\[
\|(L - L_0) u\|_{B^0} \leq \sum_{k,\ell} \|[\mathfrak{a}_{k\ell} - \mathfrak{a}_{k\ell}^0] \partial_\ell u\|_{B^1} \leq \sqrt{2}(2\pi)^{-n} \sum_{k,\ell} \|\mathfrak{a}_{k\ell} - \mathfrak{a}_{k\ell}^0\|_{B^1} \|u\|_{B^2}.
\]  

Since $ L = L_0 + (L - L_0) $, we have  
\[
L_0^{-1} L = 1 + L_0^{-1} (L - L_0).
\]  
From the preceding estimates,  
\[
\|L_0^{-1} (L - L_0)\|_{\mathscr{B}(B^0)} \leq \sqrt{2}\varkappa (2\pi)^{-n} \sum_{k,\ell} \|\mathfrak{a}_{k\ell} - \mathfrak{a}_{k\ell}^0\|_{B^1} = \eta < 1.
\]  
Therefore, $ L_0^{-1} L $ is invertible, and  
\[
\|(L_0^{-1} L)^{-1}\|_{\mathscr{B}(B^0)} \leq (1 - \eta)^{-1}.
\]  
It follows that $ L: B^2 \to B^0 $ is invertible, with  
\[
\|L^{-1}\|_{\mathscr{B}(B^0, B^2)} \leq \varkappa (1 - \eta)^{-1}.
\]  

We thus state the following result.

\begin{theorem}\label{thmani}
Let Assumption \ref{assp_aniso} hold. 
For all $ f \in B^0 $, the equation  
\begin{equation}\label{ani1}
\mu u - \sum_{k,\ell} \partial_k (\mathfrak{a}_{k\ell} \partial_\ell u) = f
\end{equation}  
admits a unique solution $ u \in B^2 $. Moreover, the following estimate holds:  
\[
\|u\|_{B^2} \leq \varkappa (1 - \eta)^{-1} \|f\|_{B^0}.
\]  
\end{theorem}

\begin{remark}\label{remani0}
{\rm
Let $ V \in B^0 \setminus \{0\} $. Similar to the case of $ (1 - \Delta)^{-1} V $, we observe that $ L_0^{-1} V: B^0 \to B^0 $ is compact. Since  
\[
L^{-1} = (1 + L_0^{-1} (L - L_0))^{-1} L_0^{-1},
\]  
it follows that $ L^{-1} V: B^0 \to B^0 $ is also compact. Consequently, the spectral analysis for the problem  
\[
(L + \lambda V) u = 0, \quad \lambda \neq 0,
\]  
parallels that of $ (1 - \Delta + \lambda V) u = 0 $.  
}\end{remark}

Next, we discuss the higher regularity of the solution to equation \eqref{ani1}. Assume $ \partial_j \mathfrak{a}_{k\ell} \in B^1 $ for $ 1 \leq j, k, \ell \leq n $ and $ f \in B^1 $. Let $ u \in B^2 $ be the solution of \eqref{ani1}. We verify that $ v_j = \partial_j u $ ($ 1 \leq j \leq n $) is the unique solution in $ \mathscr{S}' $ of the equation  
\begin{align}\label{eq_Loperator}
L v_j = \sum_{k,\ell} \partial_k (\partial_j \mathfrak{a}_{k\ell} \partial_\ell u) + \partial_j f.
\end{align}
By Theorem \ref{thmani}, $ v_j \in B^2 $, and there exists a constant $ \mathbf{c} = \mathbf{c}(n,\mathfrak{A}, \kappa, \mu) >0$ such that  
\[
\|v_j\|_{B^2} \leq \mathbf{c} \|f\|_{B^1}.
\]  
Thus, $ u \in B^3 $, and by replacing $ \mathbf{c} $ with a similar constant, we obtain  
\[
\|u\|_{B^3} \leq \mathbf{c} \|f\|_{B^1}.
\]  
By induction on an integer $ m \geq 0 $, we deduce $ B^{2 + m} $ regularity for the solution of \eqref{ani1}. Specifically, we have the following result.

\begin{theorem}\label{thmani2}
Let $ m \geq 0 $ be an integer and Assumption \ref{assp_aniso} holds. Moreover, assume $ \partial^\alpha \mathfrak{a}_{k\ell} \in B^1 $ for all $ 1 \leq |\alpha| \leq m $ when $ m \geq 1 $. For all $ f \in B^m $, the equation \eqref{ani1} admits a unique solution $ u \in B^{2 + m} $ satisfying  
\[
\|u\|_{B^{2 + m}} \leq \mathbf{c} \|f\|_{B^m},
\]  
where $ \mathbf{c} = \mathbf{c}(n, \mathfrak{A}, \kappa, \mu, m) > 0 $ is a constant.  
\end{theorem}

\begin{remark}\label{remani1}
{\rm
$\mathrm{(i)}$ Theorem \ref{thmani} remains valid if $ L $ is replaced by the operator  
\[
\tilde{L} = L + \sum_k \mathfrak{b}_k \partial_k,
\]  
where $ \mathfrak{b}_k \in B^0 $ for $ 1 \leq k \leq n $, and the condition \eqref{ani0} is modified to  
\[
\eta := \varkappa (2\pi)^{-n} \left[ 2 \sum_{k,\ell} \|\mathfrak{a}_{k\ell} - \mathfrak{a}_{k\ell}^0\|_{B^1} + \sum_k \|\mathfrak{b}_k\|_{B^0} \right] < 1.
\]  
Theorem \ref{thmani2} can also be extended to $ \tilde{L} $, provided $ \mathfrak{b}_k \in B^m $ and the constant $ \mathbf{c} $ depends on $ (\mathfrak{b}_1, \ldots, \mathfrak{b}_n) $.  
\\
$\mathrm{(ii)}$ Let $s>0$, $ V \in B^s $ satisfy $ V \geq -\mu $, and let $ u \in B^2 $ satisfy $ (L_0 + V) u = 0 $. Define $ v(y) = u(\sqrt{\mathfrak{A}^0} \, y) $ and $ W(y) = V(\sqrt{\mathfrak{A}^0} \, y) $. We verify that $ (\mu - \Delta + W) v = 0 $. Using the equivalence of norms induced by $ |\sqrt{\mathfrak{A}^0} \cdot| $ on $ \mathbb{R}^n $, we adapt the proof of \cite[Proposition 3.8]{CLLZ} to conclude $ v = 0 $. By compactness with Proposition \ref{procm}, $ L_0 + V $ is invertible. Thus, for all $ f \in B^0 $, the equation $ (L_0 + V) u = f $ admits a unique solution $ u \in B^2 $ with  
\[
\|u\|_{B^2} \leq \mathbf{c} \|f\|_{B^0},
\]  
where $ \mathbf{c} = \mathbf{c}(n, \mathfrak{A}^0, V) > 0 $ is a constant.  
}\end{remark}

To establish $ B^{s+2} $ regularity for the solution when $ s \in (0, \infty) \setminus \mathbb{N} $, we solve equation \eqref{ani1} directly in $ B^{s+2} $. Instead of Assumption \ref{assp_aniso}, we impose the following assumptions, where $ s \geq 0 $ is arbitrarily fixed. 

\begin{assumption}\label{assp_aniso2}
We assume that $\mathfrak{A}^0 $ and $\mathfrak{A}$ satisfy following conditions  
\begin{enumerate}
\item $ \mathfrak{a}_{k\ell} \in C_b^0 $, and $ \mathfrak{a}_{k\ell} - \mathfrak{a}_{k\ell}^0 \in B^{1+s} $ for each $ k, \ell $.  

\item There exists $ \kappa > 0 $ such that  
\[
\mathfrak{A}^0 \xi \cdot \xi \geq \kappa |\xi|^2, \quad  \xi \in \mathbb{R}^n.
\]  

\item The following inequality holds:  
\[
\eta' := 2^{(1+s)/2} \varkappa (2\pi)^{-n} \sum_{k,\ell} \|\mathfrak{a}_{k\ell} - \mathfrak{a}_{k\ell}^0\|_{B^{1+s}} < 1,
\]  
where $ \varkappa := 1 / \min(\mu, \kappa) $.  
\end{enumerate}

\end{assumption}
By adapting the proof of Theorem \ref{thmani}, we prove the following result.

\begin{theorem}\label{thmani3}
Let $ s \geq 0 $, and let Assumption \ref{assp_aniso2} hold. Then, for all $ f \in B^s $, the equation $ Lu = f $ admits a unique solution $ u \in B^{2+s} $, and the following estimate holds:  
\[
\|u\|_{B^{2+s}} \leq \varkappa (1 - \eta')^{-1} \|f\|_{B^s}.
\]  
\end{theorem}

Combining Theorem \ref{thmehs} and Theorem \ref{thmani3}, we obtain the following H\"older regularity result.

\begin{corollary}\label{corani1}
Let Assumption \ref{assp_aniso2} hold. Let $ k \geq 0 $ be an integer, and $ 0 < \theta < 1 $. For all $ 0 < \gamma < \theta $ and $ f \in B^{k+\theta} $, the equation $ Lu = f $ has a unique solution $ u \in C_b^{2+k, \gamma} $, and  
\[
\|u\|_{C_b^{2+k, \gamma}} \leq  \mathbf{c}_\ast \varkappa (1 - \eta')^{-1} \|f\|_{B^{k+\theta}},
\]  
where $  \mathbf{c}_\ast >0$ is the norm of the embedding $ B^{2+k+\theta} \hookrightarrow C_b^{2+k, \gamma} $.  
\end{corollary}

\subsection{Unique continuation of Schr\"odinger equations}\label{subse_UC}

In this subsection, we investigate the unique continuation property for the Schr\"odinger equation in the context of generic anisotropic conductivity. This property is fundamental as it enables us to explore the uniqueness and stability of various inverse problems in future work.

Let $ \mathfrak{A} = (\mathfrak{a}_{k\ell}) \in (B^1)^{n \times n} $ be a real-valued matrix such that $ (\mathfrak{a}_{k\ell}(x)) $ is symmetric for each $ x \in \mathbb{R}^n $. Moreover, for all compact subsets $ \Omega \Subset \mathbb{R}^n $, there exists $ \kappa_\Omega > 0 $ such that  
\[
\mathfrak{A}(x) \xi \cdot \xi \geq \kappa_\Omega |\xi|^2,\quad  \xi \in \mathbb{R}^n, \; x \in \overline{\Omega}.
\]  
Let $ \mathfrak{B} = (\mathfrak{b}_k) \in (B^0)^n $, $ V \in B^0 $, and $ \mu \in \mathbb{C} $. Suppose $ u \in B^2 $ is a solution of the equation  
\begin{equation}\label{uce}
-\sum_{k,\ell} \partial_\ell (\mathfrak{a}_{k\ell} \partial_k u) + \sum_k \mathfrak{b}_k \partial_k u + (V + \mu) u = 0
\end{equation}  
satisfying $ u = 0 $ in $ \omega \Subset \mathbb{R}^n $. Choose a compact subset $ \Omega \Subset \mathbb{R}^n $ and let $ D $ be a bounded domain in $ \mathbb{R}^n $ containing both $ \omega $ and $ \Omega $. Then $ u \in C^2(\overline{D}) $, $ \mathfrak{a}_{k\ell} \in C^1(\overline{D}) $, $ \mathfrak{b}_k \in C^0(\overline{D}) $, $ V \in C^0(\overline{D}) $,  
and  
\[
-\sum_{k,\ell} \partial_\ell (\mathfrak{a}_{k\ell} \partial_k u) + \sum_k \mathfrak{b}_k \partial_k u + (V + \mu) u = 0 \quad \text{in } \overline{D}.
\]  
By the unique continuation property for elliptic operators (e.g., \cite[Theorem 2.22]{Ch16}), $ u = 0 $ in $ D $, and hence $ u = 0 $ in $ \Omega $. Since $ \Omega \Subset \mathbb{R}^n $ was arbitrarily chosen, we deduce that $ u $ is identically zero.  

Next, let $ \Omega $ be a bounded Lipschitz domain in $ \mathbb{R}^n $ and $ \Gamma $ be a non-empty open subset of $ \partial \Omega $. Denote by $ \partial_\nu $ the derivative along the unit exterior normal vector field $ \nu $ to $ \Gamma $. By the unique continuation property from Cauchy data (e.g., \cite[Corollary 2.23]{Ch16}), if $ u \in B^2 $ is a solution of \eqref{uce} satisfying $ u = \partial_\nu u = 0 $ on $ \Gamma $, then $ u = 0 $.

\section{BVPs in spectral Barron spaces on bounded open sets}\label{se_BVP}

In this subsection, we conduct a rigorous and in-depth investigation of boundary value problems (BVPs) within the framework of spectral Barron spaces defined on bounded domains. 
Specifically, our analysis focuses on two key aspects: (i) establishing the well-posedness of these BVPs under appropriate conditions, and (ii) introducing novel analytical tools tailored to spectral Barron spaces.

Unless explicitly stated otherwise, the symbol $\Omega$ will denote an arbitrary but fixed bounded open subset of $\mathbb{R}^n$.

\subsection{Spectral Barron spaces $B^s(\Omega)$ with $s\ge 0$}

For $s \geq 0$, we define the closed subspace $F_s$ of $B^s$ as follows:
\begin{equation}\label{Fs}
F_s := \{f \in B^s ;\; \mathrm{supp}(f) \subset \mathbb{R}^n \setminus \Omega\}. 
\end{equation}
We then introduce the quotient space $B^s(\Omega):= B^s / F_s$.

In the sequel, $\pi_s: B^s \rightarrow B^s(\Omega)$ denotes the standard quotient map. For any $g_1, g_2 \in B^s$, the equality $\pi_s(g_1)=\pi_s(g_2)$ implies that $g_1 - g_2 \in F_s$, or equivalently, $g_1|_{\overline{\Omega}} = g_2|_{\overline{\Omega}}$.

We endow $B^s(\Omega)$ with the quotient norm defined by
\[
\|f\|_{B^s(\Omega)} = \inf \{\|g\|_{B^s};\; \pi_s(g) = f\},\quad f \in B^s(\Omega).
\]
As a result, $B^s(\Omega)$ constitutes a Banach space with respect to the norm $\|\cdot\|_{B^s(\Omega)}$.

By referring to \cite[Subsection 4.8]{Ru}, the dual space $B^s(\Omega)'$ of $B^s(\Omega)$ can be identified with
\[
F_s^\perp = \{g \in \tilde{B}^{-s};\; (g, f) = 0,\; f \in F_s\}.
\]
Here, the duality pairing $(g, f)$ is given by
\[
(g, f) = \int_{\mathbb{R}^n} \left[\langle \xi \rangle^{-s} \mathscr{F}^{-1}g\right] \left[\langle \xi \rangle^s \hat{f}\right] d\xi.
\]

With a slight abuse of notation, we often make the following identification:
\[
B^s(\Omega) = \{f = g|_{\overline{\Omega}};\; g \in B^s\}
\]
endowed with the norm
\[
\|f\|_{B^s(\Omega)}=\inf \{\|g\|_{B^s};\; g|_{\overline{\Omega}}=f\}.
\]
In what follows, we adopt the convention $C_b^{k,0}(\cdot) = C_b^k(\cdot)$. Let $s \geq 0$, $f \in B^s(\Omega)$, and $g\in B^s$ such that $g|_{\overline{\Omega}}=f$. Since $B^s \hookrightarrow C_b^{\lfloor s \rfloor, s - \lfloor s \rfloor}$, we have $g \in C_b^{\lfloor s \rfloor, s - \lfloor s \rfloor}$ and
\[
\|g\|_{C_b^{\lfloor s \rfloor, s - \lfloor s \rfloor}} \leq \mathbf{c} \|g\|_{B^s},
\]
where $\mathbf{c} = \mathbf{c}(n, s) > 0$ is a constant. Consequently, $f = g|_{\overline{\Omega}} \in C^{\lfloor s \rfloor, s - \lfloor s \rfloor}(\overline{\Omega})$ and
\[
\|f\|_{C^{\lfloor s \rfloor, s - \lfloor s \rfloor}(\overline{\Omega})} \leq \|g\|_{C_b^{\lfloor s \rfloor, s - \lfloor s \rfloor}} \leq \mathbf{c} \|g\|_{B^s},
\]
which implies
\[
\|f\|_{C^{\lfloor s \rfloor, s - \lfloor s \rfloor}(\overline{\Omega})} \leq \mathbf{c} \|f\|_{B^s(\Omega)}.
\]
Thus, we have the continuous embedding $B^s(\Omega) \hookrightarrow C^{\lfloor s \rfloor, s - \lfloor s \rfloor}(\overline{\Omega})$.

Analogously to  Proposition \ref{pro2},  we have the following results.

\begin{proposition}\label{probd1}
$\mathrm{(i)}$ Let $0 \leq s \leq t$. Then $B^t(\Omega) \hookrightarrow B^s(\Omega)$ with
\[
\|f\|_{B^s(\Omega)} \leq \|f\|_{B^t(\Omega)},\quad f\in B^t.
\]
$\mathrm{(ii)}$ For all $s \geq 0$, the space $C^\infty (\overline{\Omega})$ is dense in $B^s(\Omega)$.
\\
$\mathrm{(iii)}$ Let $s \geq 0$ and $t > s + \frac{n}{2}$ and assume $\Omega$ is of class $C^{0,1}$. Then, $H^t(\Omega) \hookrightarrow B^s(\Omega)$.
\\
$\mathrm{(iv)}$ Let $s \geq 0$. If $f, g \in B^s(\Omega)$, then $fg \in B^s(\Omega)$ and
\[
\|fg\|_{B^s(\Omega)} \leq 2^{s/2}(2\pi)^{-n} \|f\|_{B^s(\Omega)} \|g\|_{B^s(\Omega)}.
\]
\end{proposition}

\begin{proof}
(i) Let $f \in B^t(\Omega)$. From $B^t\hookrightarrow B^s$, we have
\[
\|f\|_{B^s(\Omega)}=\inf\{\|g\|_{B^s};\; g|_{\overline{\Omega}}=f\}\le \inf\{\|g\|_{B^t};\; g|_{\overline{\Omega}}=f\}=\|f\|_{B^t(\Omega)}.
\]

(ii) Let $s \geq 0$, $f \in B^s(\Omega)$, and $\epsilon > 0$. Choose $g\in B^s$ satisfying $g|_{\overline{\Omega}}=f$. Since $\mathscr{S}$ is dense in $B^s$, we can find $h \in \mathscr{S}$ such that $\|g - h\|_{B^s} \leq \epsilon$. Let $k:=h|_{\overline{\Omega}}\in C^\infty (\overline{\Omega})$. Using the fact that $f - k =(g - h)|_{\overline{\Omega}}$, we obtain $\|f - k\|_{B^s(\Omega)} \leq \|g - h\|_{B^s} \leq \epsilon$.

(iii) Let $f\in H^t(\Omega)$ . From \cite[Theorem 1.4.3.1]{Gr}, there exists $g\in H^t$ such that $g|_\Omega=f$ and $\|g\|_{H^t}\le \mathbf{c}\|f\|_{H^t(\Omega)}$, where $\mathbf{c}>0$ is a generic constant independent of $f$ and $g$. From Proposition \ref{pro2} (iii), we have $\|g\|_{B^s}\le \mathbf{c}\|f\|_{H^t(\Omega)}$, which implies $\|f\|_{B^s(\Omega)}\le \mathbf{c}\|f\|_{H^t(\Omega)}$.

(iv) Let $\epsilon>0$ and $h,k\in B^s$ such that $\|h\|_{B^s}\le\|f\|_{B^s(\Omega)}+\epsilon$ and $\|k\|_{B^s}\le\|g\|_{B^s(\Omega)}+\epsilon$. Since $hk\in B^s$ and $fg=(hk)|_{\overline{\Omega}}$, we obtain $fg\in B^s(\Omega)$ and Proposition \ref{pro2} (iv) yields
\[
\|fg\|_{B^s(\Omega)} \le \|hk\|_{B^s}\le 2^{s/2}(2\pi)^{-n} (\|f\|_{B^s(\Omega)}+\epsilon)(\|g\|_{B^s(\Omega)}+\epsilon).
\]
Letting $\epsilon\to 0$ completes the proof.
\end{proof}

On a bounded open subset $\Omega$, we enhance the embedding result in Theorem \ref{thmcem} as follows. 
\begin{theorem}\label{thmbd1}
For all $ 0 \leq s < t $, the embedding $ B^t(\Omega) \hookrightarrow B^s(\Omega) $ is compact.
\end{theorem}

\begin{proof}
Let $ \chi \in C_0^\infty $ be a fixed function such that $ \chi = 1 $ in a neighborhood of $ \overline{\Omega} $, and set $ K = \mathrm{supp}(\chi) $. Consider a bounded sequence $ (f_j) $ in $ B^t(\Omega) $ with $ \sup_j \|f_j\|_{B^t(\Omega)} \leq 1 $. For each $ j $, there exists $ g_j \in B^t $ such that $ f_j=g_j|_{\overline{\Omega}}=(\chi g_j)|_{\overline{\Omega}} $ and $ \|g_j\|_{B^t} \leq 2 $.

Since $ (\chi g_j) $ is a bounded sequence in $ B_K^t $, and by Theorem \ref{thmcem}, the embedding $ B_K^t \hookrightarrow B^s $ is compact, there exists a subsequence (still denoted $ (\chi g_j) $) that converges in $ B^s $. Consequently, for all $ j, k $, we have
\[
\|f_j - f_k\|_{B^s(\Omega)} \leq \|\chi g_j - \chi g_k\|_{B^s}.
\]
This implies that $ (f_j) $ is a Cauchy sequence in $ B^s(\Omega) $, and hence it converges in $ B^s(\Omega) $. The proof is complete.
\end{proof}

As an immediate consequence of Theorem \ref{thmbd1}, we obtain the following result.

\begin{corollary}\label{corbd1}
For all $ 0 \leq s < t $, the embedding $ B^t(\Omega) \hookrightarrow C^{\lfloor s \rfloor, s - \lfloor s \rfloor}(\overline{\Omega}) $ is compact. In particular, for any integer $ k \geq 0 $ and $ 0 < \gamma < \theta < 1 $, the embedding $ B^{k+\theta}(\Omega) \hookrightarrow C^{k, \gamma}(\overline{\Omega}) $ is compact.
\end{corollary}

With reference to the proof of Corollary \ref{CR}, we further obtain the following result.
\begin{corollary}\label{CR2}
For all $p\ge 1$, an integer $k\ge 0$, and for all $0<\theta<1$ and $0<\sigma <\frac{\theta}{p}$, $B^{k+\theta}(\Omega)$ is embedded in $W^{k+\sigma,p}(\Omega)$.
\end{corollary}

In what follows, let $ \Gamma = \partial \Omega $, and for $ s \geq 0 $, define the trace space
\[
B^s(\Gamma) := \{ h = u|_{\Gamma} ;\; u \in B^s(\Omega) \},
\]
equipped with the quotient norm
\[
\|h\|_{B^s(\Gamma)} := \inf \{ \|u\|_{B^s(\Omega)} ;\; u|_{\Gamma} = h \}.
\]

For further applications, we note that as an immediate consequence of Proposition \ref{probd1}, the space $ C^\infty(\Gamma) $ is dense in $ B^s(\Gamma) $ for all $ s \geq 0 $. Here, $C^\infty(\Gamma):=\{h=\varphi|_\Gamma;\; \varphi\in C^\infty_0\}$.

\begin{theorem}\label{thmtr1}
For all $ 0 \leq s < t $, the embedding $ B^t(\Gamma) \hookrightarrow B^s(\Gamma) $ is compact.
\end{theorem}

\begin{proof}
Let  $ 0 \leq s < t $ and $(h_j)$ be a sequence in $B^t(\Gamma)$ satisfying $\sup_j\|h_j\|_{B^t(\Gamma)}\le 1$. For each $j$, we find $u_j\in B^t(\Omega)$ such that $u_j|_\Gamma=h_j$ and $\|u_j\|_{B^t(\Omega)}\le 2$. By the compactness of the embedding $ B^t(\Omega) \hookrightarrow B^s(\Omega) $, upon subtracting a subsequence, we can assume that $(u_j)$ converges in $B^s(\Omega)$. Now, for all $j,k$, we have
\[
\|h_j-h_k\|_{B^s(\Gamma)}\le \|u_j-u_k\|_{B^s(\Omega)}.
\]
Therefore, $(h_j)$ is a Cauchy sequence in $B^s(\Gamma)$ and hence it converges in $B^s(\Gamma)$.
\end{proof}

\subsection{BVPs in $B^s(\Omega)$ with $s\ge 0$}\label{subse_BVP}

After the formal definition of the spectral Barron spaces $ B^s(\Omega) $, we proceed to investigate BVPs in these spaces. 
For $ s \geq 0 $ and $ (f, h) \in B^s(\Omega) \times B^{s+2}(\Gamma) $, consider the BVP
\begin{equation}\label{bvp1}
(1 - \Delta) u = f \quad \text{in } \Omega, \quad u|_{\Gamma} = h.
\end{equation}
This BVP has at most one solution in $ B^{s+2}(\Omega) $, as a consequence of the fact that $ B^{s+2}(\Omega) \hookrightarrow H^2(\Omega) $ and that \eqref{bvp1} has at most one solution in $ H^2(\Omega) $.

Before stating the existence result for \eqref{bvp1}, we introduce a definition. For all $0\le s< t  $ and $\varrho>0$, let $ \mathcal{D}_\varrho^{s,t} $ be the closure in $ B^s(\Omega) \times B^{2+s}(\Gamma) $ of the set
\[
D_\varrho^t := \{(f, h) = ((1 - \Delta) u, u|_{\Gamma}) ;\; u \in B^{2+t}(\Omega),\; \|u\|_{B^{t+2}(\Omega)}\le \varrho\}.
\]

\begin{theorem}\label{thmbd2}
Let $ 0 \leq s < t $, $\varrho >0$ and $ (f, h) \in \mathcal{D}_\varrho^{s,t} $. Then \eqref{bvp1} has a unique solution $ u \in B^{2+s}(\Omega) $ satisfying $ \|u\|_{B^{2+s}(\Omega)}\le \varrho$.
\end{theorem}

\begin{proof}
We only need to prove existence. Let $ ((f_j, h_j)) = (((1 - \Delta) u_j, u_j|_{\Gamma})) $ be a sequence in $ D_\varrho^t $ converging to $ (f, h) $ in $ B^s(\Omega) \times B^{s+2}(\Gamma) $, where $ (u_j) $ is a sequence in $ B^{2+t}(\Omega) $ satisfying $\sup_j\|u_j\|_{B^{2+t}(\Omega)}\le \varrho$. Using that the embedding $B^{2+t}(\Omega)\hookrightarrow B^{2+s}(\Omega)$ is compact by Theorem \ref{thmbd1} and subtracting a subsequence if necessary, we can assume that $(u_j)$ converges in $B^{2+s}(\Omega)$ to $u\in B^{2+s}(\Omega)$. In consequence, $\|u\|_{B^{2+s}(\Omega)}\le \varrho$ and $ (((1 - \Delta) u_j, u_j|_{\Gamma})) $ converges in $ B^s(\Omega) \times B^{2+s}(\Gamma) $ to $ ((1 - \Delta) u, u|_{\Gamma}) $. By the uniqueness of the limit, we get $ ((1 - \Delta) u, u|_{\Gamma})=(f,h) $. This completes the proof.
\end{proof}

Theorem \ref{thmbd2} can be extended to the case where $ 1 - \Delta $ is replaced by $ 1 - \Delta + \lambda V $ with $V\in B^t(\Omega)$ for some $t> 0$. Before going into details we introduce some new definitions. For simplicity, we assume in the following that $ \Omega $ is at least of class $ C^{1,1} $. Let $ \Delta_D $ be the operator defined by $ \Delta_D u = \Delta u $ for $ u \in D(\Delta_D) = H_0^1(\Omega) \cap H^2(\Omega) $. We recall that this operator has a compact resolvent. Therefore, for $ V \in B^0(\Omega)\hookrightarrow C^0(\overline{\Omega}) $, the operator $ S_V = (1 - \Delta_D)^{-1}V: L^2(\Omega) \rightarrow L^2(\Omega) $ is compact.

For $ 0\le s<t$, $ \varrho >0$ and $ V \in B^t(\Omega) $, let $ \mathcal{D}_{V,\varrho}^{s,t} $ as the closure in $ B^s(\Omega) \times B^{2+s}(\Gamma) $ of the set
\[
D_{V,\varrho}^t = \left\{ (f, h) = \left( (1 - \Delta + V)u, u|_{\Gamma} \right) ;\; u \in B^{2+t}(\Omega),\; \|u\|_{B^{2+t}(\Omega)}\le \varrho \right\}.
\]

\begin{theorem}\label{thmbd3}
Let $ 0 \leq s < t $, $\varrho>0$ and $ V \in B^t(\Omega) \setminus \{0\} $, and assume that $ -\lambda^{-1} \not\in \sigma(S_V) $. For all $(f,h)\in \mathcal{D}_{\lambda V,\varrho}^{s,t} $, the BVP
\[
(1 - \Delta + \lambda V)u = f \quad \text{in } \Omega, \quad u|_{\Gamma} = h,
\]
admits a unique solution $ u \in B^{2+s}(\Omega) $ satisfying $\|u\|_{B^{2+s}(\Omega)}\le \varrho$.
\end{theorem}

Next, from Proposition \ref{probd1} (iii), we note that
\[
C^{k-1,1}(\overline{\Omega})\hookrightarrow W^{k,\infty}(\Omega)\hookrightarrow H^k(\Omega)\hookrightarrow B^t(\Omega)
\]
holds for $t\ge 0$ and an integer $k$ satisfying $k>t+n/2$.
The following result provides a useful consequence of Theorem \ref{thmbd3}.

\begin{theorem}\label{thmbd3.1}
Let $ s\ge 0 $, $ k > s + n/2 $ be an integer, and assume that $ \Omega $ is of class $ C^{2+k} $. Let  $V=0$ or $ V \in C^{k-1,1}(\overline{\Omega}) \setminus \{0\} $ and $ -\lambda^{-1} \notin \sigma(S_V) $. For all $ (f, h) \in H^k(\Omega) \times H^{k+3/2}(\Gamma) $, the BVP
\begin{equation}\label{bvp0}
(1 - \Delta + \lambda V)u = f \quad \text{in } \Omega, \quad u|_{\Gamma} = h,
\end{equation}
admits a unique solution $ u \in B^{2+s}(\Omega) $, and the following inequality holds:
\[
\|u\|_{B^{2+s}(\Omega)} \leq \mathbf{c} \left( \|f\|_{H^k(\Omega)} + \|h\|_{H^{k+3/2}(\Gamma)} \right).
\]
Here, $ \mathbf{c} = \mathbf{c}(n, \Omega, k, s, V, \lambda) > 0 $ is a constant.
\end{theorem}

\begin{proof}
Assume that $V\ne 0$. Then the case $(f,h)=0$ is trivial. Let $ 0\neq (f, h) \in H^k(\Omega) \times H^{k+3/2}(\Gamma)$. Since $ -\lambda^{-1} \notin \sigma(S_V) $, the BVP \eqref{bvp0} admits a unique solution $ u \in H^2(\Omega) $. Using the fact that $ V \in C^{k-1,1}(\overline{\Omega}) $, we deduce from \cite[Theorem 8.13]{GT} that $ u \in H^{2+k}(\Omega) $ and
\begin{equation}\label{bvp0.1}
\|u\|_{H^{2+k}(\Omega)} \leq \mathbf{c}_0 \left( \|f\|_{H^k(\Omega)} + \|h\|_{H^{3/2+k}(\Gamma)} \right).
\end{equation}
Here and henceforth, $ \mathbf{c}_0 = \mathbf{c}_0(n, \Omega, k, V, \lambda) > 0 $ is a generic constant.

On the other hand, from \cite[Theorem 1.4.3.1]{Gr}, there exists $ g \in H^{2+k} $ such that $ g|_{\Omega} = u $ and
\[
\|g\|_{H^{2+k}} \leq \mathbf{c}_1 \|u\|_{H^{2+k}(\Omega)}
\]
for some constant $ \mathbf{c}_1 = \mathbf{c}_1(n, \Omega, k) > 0 $. Combined with \eqref{bvp0.1}, this inequality yields
\begin{equation}\label{bvp0.2}
\|g\|_{H^{2+k}} \leq \mathbf{c}_0 \left( \|f\|_{H^k(\Omega)} + \|h\|_{H^{3/2+k}(\Gamma)} \right).
\end{equation}
Take $s<t<k-n/2$. From Proposition \ref{pro2} (iii), we know that $ H^{2+k} \hookrightarrow B^{2+t} $. Thus, $ g \in B^{2+t} $, and from \eqref{bvp0.2}, we obtain
\begin{equation}\label{bvp0.3}
\|g\|_{B^{2+t}} \leq \mathbf{c}_2 \left( \|f\|_{H^k(\Omega)} + \|h\|_{H^{3/2+k}(\Gamma)} \right)=:\varrho\; (>0),
\end{equation}
where $ \mathbf{c}_2 = \mathbf{c}_2(n, \Omega, k,  s, V, \lambda) > 0 $ is a constant.

Now, since $ (f, h) = ((1 - \Delta + \lambda V)g|_{\overline{\Omega}}, g|_{\Gamma}) $, we have $ (f, h) \in D^t_{\lambda V,\varrho} $.
The proof is completed for the case $V\in B^t(\Omega)\setminus\{0\}$ by applying Theorem \ref{thmbd3}. The proof in the case $V=0$ is similar.
\end{proof}

\begin{remark}\label{rembvp1}
{\rm
We observe, from the previous proof, that the BVP \eqref{bvp0} admits a unique solution $u\in B^{2+s}(\Omega)$ for all $(f,h)$ in the closure of $H^k(\Omega) \times H^{k+3/2}(\Gamma)$ in $B^s(\Omega)\times B^{2+s}(\Gamma)$. But the a priori estimate \eqref{bvp0.1} is not necessarily verified in this case.
}
\end{remark}

Let $ L $ be the operator defined in (\ref{opeL}). We assume that $ L $ satisfies Assumption \ref{assp_aniso} when $ s \in \mathbb{N} $, and Assumption \ref{assp_aniso2} when $ s \in (0, \infty) \setminus \mathbb{N} $. In view of Theorems \ref{thmani}, \ref{thmani2}, and \ref{thmani3}, as well as Remark \ref{remani0}, it becomes evident that all the results presented in this section extend to the scenario where $ 1 - \Delta $ is replaced by $ L $ within the anisotropic setting, with the constants in the estimates also depending on the coefficients of $ L $.

\subsection{Smooth perturbations of BVPs}

In this subsection, we focus on the well-posedness of BVPs for the smooth perturbations of the operator $1-\Delta +\lambda V$. To achieve this, we fix $0 \leq s < t < \min(s + 2, k-n/2)$, $V \in C^{k-1,1}(\overline{\Omega}) \setminus \{0\}$ and $-\lambda^{-1} \not\in \sigma(S_V)$ or $V=0$. Assume that $\Omega$ is of class $C^{2+k}$. According to Theorem \ref{thmbd3.1}, for all $f \in H^k(\Omega)$ and $h \in H^{k+3/2}(\Gamma)$, the BVP
\begin{equation}\label{se1}
(1 - \Delta + \lambda V)u = f \quad \text{in } \Omega, \quad u|_{\Gamma} = h
\end{equation}
admits a unique solution $u=u(f, h) \in B^{2+s}(\Omega)$ satisfying
\begin{equation}\label{se2}
\|u(f, h)\|_{B^{2+s}(\Omega)} \leq \mathbf{c}\left(\|f\|_{H^k(\Omega)} + \|h\|_{H^{k+3/2}(\Gamma)}\right),
\end{equation}
where $\mathbf{c} = \mathbf{c}(n,\Omega, k, s, t, V, \lambda) > 0$ is a constant. We define
\[
\mathbf{\hat{c}} := \sup\left\{ \|u(f, h)\|_{B^{2+s}(\Omega)}; \; \|f\|_{H^k(\Omega)} + \|h\|_{H^{k+3/2}(\Gamma)} \leq 1 \right\}.
\]

Next, we fix $(f_0, h_0) \in H^k(\Omega) \times H^{k+3/2}(\Gamma)$ and let $\mathfrak{a}: B^t(\Omega) \to H^k(\Omega)$ be such that there exist $0 < \eta < \mathbf{\hat{c}}^{-1}$ and $\gamma > 0$ satisfying
\begin{equation}\label{sem}
\|\mathfrak{a}(u)\|_{H^k(\Omega)} \leq \eta \|u\|_{B^t(\Omega)} + \gamma, \quad u \in B^t(\Omega).
\end{equation}

We define
\[
\Psi: (u, \tau) \in B^t(\Omega) \times [0, 1] \mapsto \tau u(\mathfrak{a}(u) + f_0, h_0) \in B^t(\Omega).
\]
Let $(u, \tau) \in B^t(\Omega) \times [0, 1]$ such that $u = \Psi(u, \tau)$. We have
\begin{align}
\|u\|_{B^t(\Omega)} \leq \|u\|_{B^{2+s}(\Omega)} &\leq \|u(\mathfrak{a}(u) + f_0, h_0)\|_{B^{2+s}(\Omega)}\label{se3} \\
&\leq \mathbf{\hat{c}}\left(\eta \|u\|_{B^t(\Omega)} + \|f_0\|_{H^k(\Omega)} + \gamma + \|h_0\|_{H^{k+3/2}(\Gamma)}\right)\nonumber
\end{align}
and hence
\[
\|u\|_{B^t(\Omega)} \leq \mathbf{\hat{c}}(1 - \eta \mathbf{\hat{c}})^{-1}\left(\gamma + \|f_0\|_{H^k(\Omega)} + \|h_0\|_{H^{k+3/2}(\Gamma)}\right).
\]
On the other hand, since the embedding $B^{2+s}(\Omega) \hookrightarrow B^t(\Omega)$ is compact and
\[
\|\Psi(u, \tau)\|_{B^{2+s}(\Omega)} \leq \mathbf{\hat{c}}\left(\eta + \gamma + \|f_0\|_{H^k(\Omega)} + \|h_0\|_{H^{k+3/2}(\Gamma)}\right)
\]
for all $\tau \in [0, 1]$ and $\|u\|_{B^t(\Omega)} \le 1$, we deduce that $\Psi$ is a compact operator. Since $\Psi(\cdot, 0) = 0$, by applying the Leray-Schauder fixed-point theorem (e.g., \cite[Theorem 11.6]{GT}), we conclude that $\Psi(\cdot, 1)$ admits a fixed point $u^* \in B^t(\Omega)$. Whence, $u^* \in B^{2+s}(\Omega)$ and it  is a solution to the BVP
\begin{equation}\label{se4}
(1 - \Delta + \lambda V)u = \mathfrak{a}(u) + f_0 \quad \text{in } \Omega, \quad u|_{\Gamma} = h_0.
\end{equation}
From \eqref{se3}, it follows that
\begin{equation}\label{se5}
\|u^*\|_{B^{2+s}(\Omega)} \leq \mathbf{\hat{c}}(1 - \eta \mathbf{\hat{c}})^{-1}\left(\gamma + \|f_0\|_{H^k(\Omega)} + \|h_0\|_{H^{k+3/2}(\Gamma)}\right).
\end{equation}

\begin{theorem}\label{thmse}
Under the additional assumptions that $1 + \lambda V \geq 0$ in $\Omega$ and
\[
(\mathfrak{a}(u) - \mathfrak{a}(v))(\overline{u} - \overline{v}) \leq 0, \quad u, v \in B^t(\Omega),
\]
the BVP \eqref{se4} admits a unique solution $u^* \in B^{2+s}(\Omega)$ satisfying \eqref{se5}.
\end{theorem}

\begin{proof}
We have already established the existence of a solution. What remains is to prove uniqueness. Let $u_1, u_2 \in B^{2+s}(\Omega)$ be two solutions of \eqref{se4}. Then, we have
\[
(1 - \Delta + \lambda V)(u_1 - u_2) = \mathfrak{a}(u_1) - \mathfrak{a}(u_2) \quad \text{in } \Omega, \quad (u_1 - u_2)|_{\Gamma} = 0.
\]
Applying Green's formula, we obtain
\begin{align*}
&\int_{\Omega} \left[(1 + \lambda V)|u_1 - u_2|^2 + |\nabla (u_1 - u_2)|^2\right] dx 
\\
&\hskip 4cm = \int_{\Omega} (\mathfrak{a}(u_1) - \mathfrak{a}(u_2))(\overline{u}_1 - \overline{u}_2) dx \leq 0.
\end{align*}
This implies that $u_1 = u_2$.
\end{proof}

\begin{remark}\label{remse1}
{\rm
If we replace condition \eqref{sem} with the following: there exists $0 < \eta < \mathbf{\hat{c}}^{-1}$ such that
\begin{equation}\label{sem1}
\|\mathfrak{a}(u) - \mathfrak{a}(v)\|_{H^k(\Omega)} \leq \eta \|u - v\|_{B^t(\Omega)}, \quad u, v \in B^t(\Omega),
\end{equation}
then we can verify that the mapping
\[
\Phi: u \in B^t(\Omega) \mapsto u(\mathfrak{a}(u) + f_0, h_0) \in B^t(\Omega)
\]
is contractive. Consequently, by the Banach contraction principle, $\Phi$ admits a unique fixed point $u^*$. Since $\Phi(u) \in B^{2+s}(\Omega)$, we deduce that $u^* \in B^{2+s}(\Omega)$ and is the unique solution of \eqref{se4}. Moreover, as \eqref{sem1} implies \eqref{sem} with $\gamma = \|\mathfrak{a}(0)\|_{H^k(\Omega)}$, $u^*$ satisfies \eqref{se5} with $\gamma = \|\mathfrak{a}(0)\|_{H^k(\Omega)}$.
}\end{remark}

Let us give an example of a function $\mathfrak{a}$ satisfying \eqref{sem1}. \cite[Theorem 8.13]{GT} shows that for all integer $m\ge 0$ and $f\in H^m(\Omega)$, the BVP \eqref{se1} admits a unique solution $U(f):=(1-\Delta_D)^{-1}f\in H^{2+m}(\Omega)$ provided that $\Omega$ is of class $C^{2+m}$. Furthermore, $\|U(f)\|_{H^{m+2}(\Omega)}\le \mathbf{c}_0\|f\|_{H^m(\Omega)}$, where $\mathbf{c}_0$ is a constant independent of $f$. The particular choice of $t=m$ and $m=k-2$ yields $\|U(f)\|_{H^k(\Omega)}\le \mathbf{c} \|f\|_{B^t(\Omega)}$, where the constant $\mathbf{c} >0$ is independent of $f$. Here, we used that $B^m(\Omega)\hookrightarrow C^m(\overline{\Omega})\hookrightarrow H^m(\Omega)$. Let us observe that the condition $k>t+n/2$ holds only when $n=2$ or $n=3$. For $f_1\in H^k(\Omega)$ we verify that $\mathfrak{a}$ given by
\[
\mathfrak{a}(u):=\delta (1-\Delta_D)^{-1}u+f_1,\quad u\in B^t(\Omega),
\]
satisfies \eqref{sem1} provided that $\eta:=\delta \mathbf{c}<\mathbf{\hat{c}}^{-1}$.

\section*{Acknowledgement}
This work was supported by National Key Research and Development Programs of China (No. 2023YFA1009103), JSPS KAKENHI Grant Numbers JP25K17280, JP23KK0049, NSFC (No. 92570106) and Science and Technology Commission of Shanghai Municipality (No. 23JC1400501).


\end{document}